\documentclass[dvipsnames,a4paper]{amsart}

\author[N. J. Williams]{Nicholas J. Williams}

\address{Department of Mathematics and Statistics, Fylde College, Lancaster University, Lancaster, LA1 4YF, United Kingdom}
\email{nicholas.williams@lancaster.ac.uk}

\usepackage[backend=biber,
	style=alphabetic,
	url=false,
	doi=false,
	isbn=false,
	maxbibnames=99]{biblatex}
\DeclareFieldFormat{title}{#1} 
\DeclareFieldFormat{postnote}{#1} 
\renewbibmacro{in:}{} 
\theoremstyle{plain}
 \newtheorem{theorem}{Theorem}[section]
 
 \newtheorem{proposition}[theorem]{Proposition}
 \newtheorem{corollary}[theorem]{Corollary}
 \newtheorem{innercustomthm}{Theorem}
 \newenvironment{customthm}[1]
  {\renewcommand\theinnercustomthm{#1}\innercustomthm}
  {\endinnercustomthm}

\theoremstyle{definition}
 \newtheorem{remark}[theorem]{Remark}
 
 \newtheorem{example}[theorem]{Example}

\usepackage{tikz}
 \usetikzlibrary{decorations.pathmorphing}
 \usetikzlibrary{decorations.pathreplacing}
 \usetikzlibrary{positioning}
 \usetikzlibrary{calc}
 \usetikzlibrary{intersections}
 \usetikzlibrary{fit}
\usepackage{tikz-cd} 

\makeatletter
\tikzset{
    dot diameter/.store in=\dot@diameter,
    dot diameter=3pt,
    dot spacing/.store in=\dot@spacing,
    dot spacing=10pt,
    dots/.style={
        line width=\dot@diameter,
        line cap=round,
        dash pattern=on 0pt off \dot@spacing
    }
}
\makeatother

\usepackage{url}
\usepackage{amssymb}
\usepackage{mathtools}
\usepackage{tabstackengine}
 \setstackgap{L}{.7\baselineskip}
\usepackage{tabularx}
 
\newcommand{\nonconsec}[2]{\prescript{\circlearrowleft}{}{\mathbf{I}_{#1}^{#2}}}

\newcommand{\modset}[2]{{\mathbf{I}_{#1}^{#2}}}

\newcommand{\tcs}[1]{\tabbedCenterstack{#1}}
\newcommand{\floor}[1]{\lfloor #1 \rfloor}
\newcommand{\ceil}[1]{\lceil #1 \rceil}

\DeclareMathOperator{\Hom}{Hom}
\DeclareMathOperator{\Ext}{Ext}

\DeclareMathOperator{\modules}{mod}
\DeclareMathOperator{\additive}{add}
\DeclareMathOperator{\gldim}{gl.dim}

\DeclareMathOperator{\conv}{conv}

\newcommand{\clus}{\mathcal{C}}
\newcommand{\orth}[1]{\prescript{\perp}{}{#1}}

\numberwithin{table}{section}
\numberwithin{figure}{section}

\title{The higher Stasheff--Tamari orders in representation theory}


\thanks{I would like to thank the scientific committee of ICRA 2020 for selecting me to give a talk based on my research snapshot on this subject, and consequently giving me the opportunity to contribute to the proceedings of the conference. I would also like to thank my supervisor Sibylle Schroll along with Jordan McMahon, Vic Reiner, Hugh Thomas, and Hipolito Treffinger for helpful input on the paper \cite{njw-hst}. Thank you finally to an anonymous referee for helpful comments. At the time I wrote the paper \cite{njw-hst}, I was supported by a studentship from the University of Leicester.}

\begin{document}




\begin{abstract}
We show that the relationship discovered by Oppermann and Thomas between triangulations of cyclic polytopes and the higher Auslander algebras of type $A$, denoted $A_{n}^{d}$, is an incredibly rich one. The \emph{higher Stasheff--Tamari orders} are two orders on triangulations of cyclic polytopes, conjectured to be equivalent, defined in the 1990s by Kapranov and Voevodsky, and Edelman and Reiner. We first show that these orders correspond in even dimensions to natural orders on tilting modules defined by Riedtmann and Schofield and studied by Happel and Unger. This result allows us to show that triangulations of odd-dimensional cyclic polytopes are in bijection with equivalence classes of $d$-maximal green sequences of $A_{n}^{d}$, which we introduce as a higher-dimensional generalisation of the original maximal green sequences of Keller. We further interpret the higher Stasheff--Tamari orders in odd dimensions, where they correspond to natural orders on equivalences classes of $d$-maximal green sequences. The conjecture that these two partial orders on equivalence classes of $d$-maximal green sequences are equal amounts to an oriented version of the ``no-gap'' conjecture of Br\"ustle, Dupont, and Perotin. A corollary of our results is that this conjecture holds for $A_{n}$, and that here the set of equivalence classes of (1-)maximal green sequences is a lattice.
\end{abstract}

\maketitle


\section{Introduction}

An essential tool for studying the representation theory of Artin algebras is Auslander--Reiten theory, developed by Auslander and Reiten in a series of seminal papers \cite{a1,a2,ar-stable,ar3,ar4,ar5,ar6}. In subsequent pioneering work, Iyama established a higher-dimensional version of their theory, extending many results \cite{iy-high-ar,iy-aus,iy-clus}, see also \cite{iy-revis}. Since then, higher Auslander--Reiten theory has proved a vibrant and fruitful area of research which has found connections with diverse areas of mathematics, including symplectic geometry \cite{djl}, non-commutative algebraic geometry \cite{himo}, and combinatorics \cite{ot,njw-hst}.

Higher Auslander--Reiten theory studies so-called $d$-cluster-tilting subcategories of $\modules \Lambda$, which behave like higher-dimensional analogues of abelian categories \cite{jasso}. If $\additive M$ is a $d$-cluster-tilting subcategory for a module $M$, then we say that $M$ is a \emph{$d$-cluster-tilting module}. If $\Lambda$ has a $d$-cluster-tilting module and $\gldim \Lambda \leqslant d$, then $\Lambda$ is called \emph{$d$-representation-finite $d$-hereditary}. These algebras are higher-dimensional versions of representation-finite hereditary algebras. The canonical examples of $d$-representation-finite algebras are the \emph{higher Auslander algebras of type $A$}, which were introduced by Iyama in \cite{iy-clus} and are denoted $A_{n}^{d}$. We denote the unique basic $d$-cluster-tilting module of $A_{n}^{d}$ by $M^{(d,n)}$.

It is the combinatorial connections of higher Auslander--Reiten theory that will interest us here. A beautiful paper of Oppermann and Thomas \cite{ot} shows that tilting $A_{n}^{d}$-modules lying inside $\additive M^{(d, n)}$ are in bijection with triangulations of the cyclic polytope $C(n + 2d, 2d)$. Cyclic polytopes are the analogues of convex polygons in higher dimensions \cite{gale}. For this reason, cyclic polytopes possess extremal combinatorial properties \cite{mcmullen} and are found in many areas of mathematics, such as higher category theory and algebraic $K$-theory \cite{kv-poly,dk-segal,poguntke,djw}, as well as game theory \cite{svs06,svs16}. In theoretical physics they appear as examples of amplituhedra \cite{amplituhedron}.

The set of triangulations of a cyclic polytope possesses additional structure. In 1991, Kapranov and Voevodsky defined a partial order on the set of triangulations of a cyclic polytope which they called the \emph{higher Stasheff order} \cite[Definition~3.3]{kv-poly}. Edelman and Reiner then developed this further by defining two partial orders on triangulations of cyclic polytopes, known as the \emph{first} ($\leqslant_{1}$) and \emph{second} ($\leqslant_{2}$) \emph{higher Stasheff--Tamari orders} \cite{er}, where the first higher Stasheff--Tamari order is the same as the higher Stasheff order \cite[Proposition~3.3]{thomas-bst}. In their paper, Edelman and Reiner conjectured the two higher Stasheff--Tamari orders to coincide. This conjecture has since been the subject of several papers \cite{rambau,err,thomas-bst,thomas,rr,njw-equal}.

For two-dimensional cyclic polytopes, both higher Stasheff--Tamari orders are equal to the Tamari lattice \cite{tamari_thesis,tamari}. This is a ubiquitous mathematical structure which arises in many fields, such as computer science, topology, and physics \cite{tamari-festschrift}. Indeed, it was shown by Buan and Krause \cite{buan-krause} that the poset of tilting modules for the path algebra of linearly oriented $A_{n}$ is isomorphic to the Tamari lattice, see also \cite{thomas_tamari}.

The present paper serves as an extended abstract for the paper \cite{njw-hst}, which contains complete proofs of the results outlined. In this paper we focus on sketching intuitive proofs of the results and giving examples. Our first result shows that the result of Buan and Krause is the $d = 1$ case of a general theorem concerning the higher Stasheff--Tamari orders and tilting modules for $A_{n}^{d}$. See Section~\ref{subsect-high-ar} for the definition of left mutation and the category $\orth{T}$.

\begin{customthm}{A}[Theorem~\ref{thm:even:first} and Theorem~\ref{thm:even:second}]\label{thm-int-even-alg}
Let $\mathcal{T}$ and $\mathcal{T}'$ be triangulations of a $2d$-dimensional cyclic polytope corresponding to tilting modules $T$ and $T'$ over $A_{n}^{d}$, the higher Auslander algebra of type $A$. We then have that
\begin{enumerate}
\item $\mathcal{T} \lessdot_{1} \mathcal{T}'$ if and only if $T'$ is a left mutation of $T$; and
\item $\mathcal{T} \leqslant_{2} \mathcal{T}'$ if and only if $\prescript{\bot}{}T \subseteq \prescript{\bot}{}T'$.
\end{enumerate}
\end{customthm}

This is an intriguing result because these orders on tilting modules are familiar. They were first considered for $d = 1$ in \cite{rs-simp} for finite-dimensional algebras over algebraically closed fields; for Artin algebras, they were proven to have the same Hasse diagram in \cite{hu-potm}. In particular, this proves for $d = 1$ that the two orders are the same when there are finitely many tilting modules. A higher-dimensional version of this result would prove the Edelman--Reiner conjecture in even dimensions.  Analogues of these orders have been studied for silting \cite{ai} and $\tau$-tilting \cite{air}.

There is an alternative algebraic framework from \cite{ot}, where triangulations of the $2d$-dimensional cyclic polytope with $n+2d+1$ vertices correspond to so-called cluster-tilting objects for $A_{n}^{d}$. Cluster-tilting objects are the analogues in cluster categories of tilting modules. We modify the cluster-tilting framework from \cite{ot} in order to accommodate partial orders on the cluster-tilting objects. There also exists a version of Theorem~\ref{thm-int-even-alg} in this framework, which is \cite[Theorem 3.13]{njw-hst}. 

In odd dimensions we show that triangulations of cyclic polytopes correspond to equivalence classes of $d$-maximal green sequences (Theorem~\ref{thm:max_green}). Maximal green sequences were introduced in the context of Donaldson--Thomas invariants in algebraic geometry \cite{kel-green}. We define higher-dim\-en\-sional \emph{$d$-maximal green sequences} as sequences of mutations of cluster-tilting objects from the projectives to the shifted projectives. We furthermore define an equivalence relation on these objects. The detail of these points can be found in Section~\ref{sect:odd}.

\begin{customthm}{B}[Theorem~\ref{thm:max_green}]\label{thm:int_odd_bij}
There is a bijection between equivalence classes of $d$-maximal green sequences of $A_{n}^{d}$ and triangulations of the cyclic polytope $C(n + 2d + 1, 2d + 1)$.
\end{customthm}

One may then ask what orders the higher Stasheff--Tamari orders correspond to under this bijection between triangulations and equivalence classes of $d$-maximal green sequences. The orders induced here are very natural, but have not been as widely studied as those on tilting modules \cite{gorsky_phd,gorsky_note,gorsky_forthcoming}. For the higher Auslander algebras of type~$A$, the Edelman--Reiner conjecture implies the ``no-gap'' conjecture made in \cite{bdp}, cases of which were proven in \cite{g-mc,hi-no-gap}. The algebraic description of the higher Stasheff--Tamari orders in odd dimensions is as follows. For the definition of the equivalence relation see Section~\ref{sect:odd}, and for the definition of an increasing elementary polygonal deformation see Section~\ref{odd-first}.

\begin{customthm}{C}[Theorem~\ref{thm:odd:first} and Theorem~\ref{thm:odd:second}]\label{thm-int-odd-alg}
Let $\mathcal{T}$ and $\mathcal{T}'$ be triangulations of a $(2d+1)$-dimensional cyclic polytope corresponding to equivalence classes of $d$-maximal green sequences $[G]$ and $[G']$ of $A_{n}^{d}$, the higher Auslander algebra of type $A$. We then have that
\begin{enumerate}
\item $\mathcal{T} \lessdot_{1} \mathcal{T}'$ if and only if $[G']$ is an increasing elementary polygonal deformation of $[G]$; and
\item $\mathcal{T} \leqslant_{2} \mathcal{T}'$ if and only if the set of summands of $[G]$ contains the set of summands of $[G']$.
\end{enumerate}
\end{customthm}

Theorems~\ref{thm:int_odd_bij} and \ref{thm-int-odd-alg} belong to the cluster-tilting framework, because $d$-maximal green sequences are defined in terms of cluster-tilting objects. As before, there are versions of Theorems~\ref{thm:int_odd_bij} and \ref{thm-int-odd-alg} in the tilting framework, which are \cite[Theorem 5.2, Theorem 5.6, Theorem 5.18]{njw-hst}. In this introduction we state Theorem~\ref{thm-int-even-alg} in the tilting framework and Theorem~\ref{thm-int-odd-alg} in the cluster-tilting framework, because these are the most natural frameworks for the respective theorems. A corollary of Theorem~\ref{thm-int-odd-alg} is that the set of equivalence classes of maximal green sequences of linearly oriented $A_{n}$ is a lattice (Corollary~\ref{cor-a-green-lat}). This is because in dimension 3 the two orders are known to be equivalent and known to be lattices \cite{er}.

The overall philosophy one extracts from these results is that cluster-theoretic phenomena are somehow even-dimensional, as are their represent\-a\-tion-theoretic parallels in the form of tilting modules, cluster-tilting objects, and the like. Indeed, one can note the connections that have been forged between representation theory \cite{hkk,lp,ops_geometric}, and higher Auslander--Reiten theory in particular \cite{djl}, and symplectic geometry, where it is essential to work in even dimensions. The odd-dimensional analogues of these cluster-theoretic phenomena are ($d$-)maximal green sequences, which themselves possess natural partial orders.

In Section~\ref{sect:background}, we give background on the higher Stasheff--Tamari orders and higher Auslander--Reiten theory. Then, in Section~\ref{sect:even}, we explain our results linking the even-dimensional higher Stasheff--Tamari orders with orders on tilting modules. Finally, in Section~\ref{sect:odd}, we explain our results linking the odd-dimensional higher Stasheff--Tamari orders with orders on equivalence classes of $d$-maximal green sequences.

\section{Background}\label{sect:background}

In this section we give necessary information on cyclic polytopes, higher Ausland\-er--Reiten theory, and the relation between them shown in \cite{ot}.

First, we set out some notational conventions. We use $[n]$ to denote the set $\{1, \dots, n\}$. Given a $(k + 1)$-tuple $A \in [n]^{k + 1}$, unless otherwise indicated, we shall denote the entries of $A$ by $A = (a_{0}, \dots, a_{k})$. The same applies to other letters of the alphabet: the upper case letter denotes the tuple; the lower case letter is used for the entries, which are numbered starting from zero.

\subsection{Cyclic polytopes}\label{back-cyc}

Cyclic polytopes should be thought of as higher-dimen\-sional analogues of convex polygons. General introductions to this class of polytopes can be found in \cite[Lecture 0]{ziegler} and \cite[4.7]{gruenbaum}.

Recall that subset $X \subset \mathbb{R}^{n}$ is \emph{convex} if for any $x,x' \in X$, the line segment connecting $x$ and $x'$ is contained in $X$. The \emph{convex hull} $\conv(X)$ of $X$ is the smallest convex set containing $X$ or, equivalently, the intersection of all convex sets containing $X$.

The \emph{moment curve} is defined by $p(t):=(t, t^{2}, \dots , t^{\delta}) \subset \mathbb{R}^{\delta}$ for $t \in \mathbb{R}$, where $\delta \in \mathbb{N}_{\geqslant 1}$. Choose $t_{1}, \dots , t_{m} \in \mathbb{R}$ such that $t_{1} < t_{2} < \dots < t_{m}$. The convex hull $\conv\{p(t_{1}), \dots , p(t_{m})\}$ is a \emph{cyclic polytope} $C(m, \delta)$. More, generally, for $H = (h_{0}, h_{1}, \dots, h_{k}) \in [m]^{k + 1}$, we write $C(H, \delta)$ for $\conv\{p(t_{h_{0}}), p(t_{h_{1}}), \dots, p(t_{h_{k}})\}$. There is a natural projection map from $C(m, \delta)$ to $C(m, \delta - 1)$ given by forgetting the last coordinate.

\begin{figure}
\caption{The cyclic polytopes $C(6,1), C(6,2), C(6,3)$ \cite[Figure 2]{er}}
\[
\begin{tikzpicture}[scale=0.8]
\coordinate(11) at (-4,0);
\coordinate(12) at (-2.4,0);
\coordinate(13) at (-0.8,0);
\coordinate(14) at (0.8,0);
\coordinate(15) at (2.4,0);
\coordinate(16) at (4,0);

\coordinate(21) at (-4,4.5);
\coordinate(22) at (-2.4,2.9);
\coordinate(23) at (-0.8,2.1);
\coordinate(24) at (0.8,2.1);
\coordinate(25) at (2.4,2.9);
\coordinate(26) at (4,4.5);

\coordinate(31) at (-4,9);
\coordinate(32) at (-2.4,7.4);
\coordinate(33) at (-0.8,6.6);
\coordinate(34) at (0.8,6.6);
\coordinate(35) at (2.4,7.4);
\coordinate(36) at (4,9);

\draw[->,thick] (0, 6) -- (0,5.1);
\draw[->,thick] (0, 1.5) -- (0,0.6);

\node(1) at (-6,0){$C(6, 3)$};
\node(2) at (-6,4.5){$C(6, 2)$};
\node(3) at (-6,9){$C(6, 1)$};

\draw (11) -- (12) -- (13) -- (14) -- (15) -- (16);

\draw (21) -- (22) -- (23) -- (24) -- (25) -- (26) -- (21);

\draw[dotted] (31) -- (33);
\draw[dotted] (31) -- (34);
\draw[dotted] (31) -- (35);
\draw (31) -- (32) -- (33) -- (34) -- (35) -- (36) -- (31);
\draw (36) -- (34);
\draw (36) -- (33);
\draw (36) -- (32);

\node at (11){$\bullet$};
\node at (12){$\bullet$};
\node at (13){$\bullet$};
\node at (14){$\bullet$};
\node at (15){$\bullet$};
\node at (16){$\bullet$};
\node at (21){$\bullet$};
\node at (22){$\bullet$};
\node at (23){$\bullet$};
\node at (24){$\bullet$};
\node at (25){$\bullet$};
\node at (26){$\bullet$};
\node at (31){$\bullet$};
\node at (32){$\bullet$};
\node at (33){$\bullet$};
\node at (34){$\bullet$};
\node at (35){$\bullet$};
\node at (36){$\bullet$};
\end{tikzpicture}
\]
\end{figure}
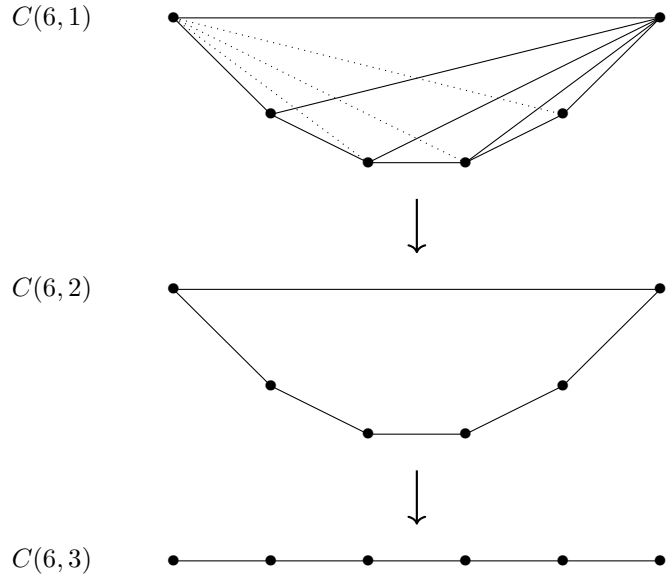

A \emph{triangulation} of a cyclic polytope $C(m, \delta)$ is a subdivision of $C(m, \delta)$ into $\delta$-dimensional simplices whose vertices are elements of $\{p(t_{1}), \dots, p(t_{m})\}$. We identify a triangulation of $C(m,\delta)$ with the corresponding set of $\delta$-simplices. We write $\mathcal{S}(m,\delta)$ for the set of triangulations of $C(m,\delta)$.

As explained in \cite{rambau}, whether or not a collection of increasing $(\delta+1)$-tuples with entries $\{t_{1}, \dots, t_{m}\}$ gives the collection of $\delta$-simplices of a triangulation is independent of the values of $t_{1} < \dots < t_{m}$ defining $C(m, \delta)$. We therefore fix $t_{i}=i \in \mathbb{R}$ throughout the paper and use $[m]$ as the vertex set of $C(m,\delta)$. One can thus describe $k$-dimensional simplices in $C(m, \delta)$ as increasing $(k + 1)$-tuples with elements in $[m]$, where $k \leqslant \delta$.  A triangulation can be specified by giving the collection of increasing $(\delta + 1)$-tuples corresponding to the $\delta$-simplices of the triangulation. Given an increasing $(k + 1)$-tuple $A \in [m]^{k + 1}$, we use $|A|_{\delta}$ to refer to the geometric $k$-simplex $\conv(A)$ in dimension $\delta$. When the dimension is clear, we will drop the subscript.

A \emph{facet} of $C(m,\delta)$ is a face of codimension one. The \emph{upper facets} of $C(m,\delta)$ are those that can be seen from a very large positive $\delta$-th coordinate. The \emph{lower facets} of $C(m,\delta)$ are those that can be seen from a very large negative $\delta$-th coordinate. Given an increasing $\delta$-tuple $F \in [m]^{\delta}$, there is a combinatorial criterion determining whether $|F|$ is an upper facet or a lower facet of $C(m, \delta)$ (or neither), known as \emph{Gale's Evenness Criterion} \cite[Theorem 3]{gale}\cite[Lemma 2.3]{er}. The collection of lower facets $\{|F|_{\delta+1}\}$ of $C(m,\delta+1)$ gives a triangulation $\{|F|_{\delta}\}$ of $C(m,\delta)$, known as the \emph{lower triangulation}. Said differently, the lower triangulation of $C(m, \delta)$ is the image of the lower facets of $C(m, \delta + 1)$ under the projection map from $C(m, \delta + 1)$ to $C(m, \delta)$. Likewise, the collection of upper facets of $C(m,\delta+1)$ gives a triangulation of $C(m,\delta)$ known as the \emph{upper triangulation}.

Indeed, every triangulation $\mathcal{T}$ of $C(m, \delta)$ determines a unique piecewise-linear section \[s_{\mathcal{T}}\colon C(m,\delta)\rightarrow C(m,\delta+1)\] of $C(m, \delta +1)$ by sending each $\delta$-simplex $|A|_{\delta}$ of $\mathcal{T}$ to $|A|_{\delta+1}$ in $C(m, \delta+1)$, in the natural way. This is a section---that is, a right inverse---of the projection map from $C(m, \delta + 1)$ to $C(m, \delta)$. A $\delta$-simplex $|A|$ in $C(m,\delta)$ similarly defines a map $s_{|A|}\colon |A|_{\delta} \rightarrow C(m,\delta+1)$. The \emph{second higher Stasheff--Tamari order} on $\mathcal{S}(m,\delta)$ is defined as \[ \mathcal{T} \leqslant_{2} \mathcal{T}' \iff s_{\mathcal{T}}(x)_{\delta+1} \leqslant s_{\mathcal{T}'}(x)_{\delta+1} ~ \forall x \in C(m, \delta),\] where $s_{\mathcal{T}}(x)_{\delta+1}$ denotes the $(\delta+1)$-th coordinate of the point $s_{\mathcal{T}}(x)$. We write $\mathcal{S}_{2}(m,\delta)$ for the poset on $\mathcal{S}(m,\delta)$ this gives.

\begin{figure}
\caption{A pair of triangulations of $C(5,1)$ such that $\mathcal{T} \leqslant_{2} \mathcal{T}'$}
\[
\begin{tikzpicture}



\coordinate (b1) at (-2,1.5);
\coordinate (b2) at (-1,1.5);
\coordinate (b3) at (0,1.5);
\coordinate (b4) at (1,1.5);
\coordinate (b5) at (2,1.5);


\draw[red] (b1) -- (b3) -- (b5);


\node at (b1) {$\bullet$};
\node at (b3) {$\bullet$};
\node at (b5) {$\bullet$};


\node at (-3,1.5) {$\mathcal{T}'$};



\coordinate (a1) at (-2,0.5);
\coordinate (a2) at (-1,0.5);
\coordinate (a3) at (0,0.5);
\coordinate (a4) at (1,0.5);
\coordinate (a5) at (2,0.5);


\draw[green] (a1) -- (a2) -- (a3) -- (a4) -- (a5);


\node at (a1) {$\bullet$};
\node at (a2) {$\bullet$};
\node at (a3) {$\bullet$};
\node at (a4) {$\bullet$};
\node at (a5) {$\bullet$};


\node at (-3,0.5) {$\mathcal{T}$};


\draw[->] (2.5, 1) -- (4, 1);



\coordinate (21) at (4, 2);
\coordinate (22) at (5, 0.5);
\coordinate (23) at (6, 0);
\coordinate (24) at (7, 0.5);
\coordinate (25) at (8, 2);


\draw[green] (21) -- (22) -- (23) -- (24) -- (25);


\draw[red] (21) -- (23) -- (25);


\draw (21) -- (25);


\node at (21) {$\bullet$};
\node at (22) {$\bullet$};
\node at (23) {$\bullet$};
\node at (24) {$\bullet$};
\node at (25) {$\bullet$};

\end{tikzpicture}
\]
\end{figure}
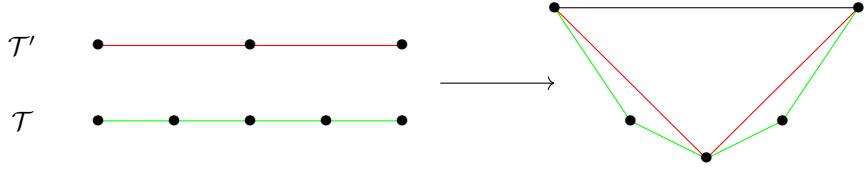

We also use the following different interpretation of the second higher Stasheff--Tamari order. A $k$-simplex $|A|$ in $C(m, \delta)$ is \emph{submerged} by the triangulation $\mathcal{T} \in \mathcal{S}(m, \delta)$ if the restriction of the piecewise linear section $s_{\mathcal{T}}$ to the simplex $|A|$ has the property that \[s_{|A|}(x)_{\delta+1} \leqslant s_{\mathcal{T}}(x)_{\delta+1}\] for all points $x \in |A|$. For a triangulation $\mathcal{T}$ of $C(m, \delta)$, the \emph{$k$-submersion set}, $\mathrm{sub}_{k}(\mathcal{T})$, is the set of $k$-simplices $|A|$ which are submerged by $\mathcal{T}$. Given two triangulations $\mathcal{T}, \mathcal{T}' \in \mathcal{S}(m,\delta)$, we have that $\mathcal{T} \leqslant_{2} \mathcal{T}'$ if and only if $\mathrm{sub}_{\lceil \frac{\delta}{2} \rceil}(\mathcal{T}) \subseteq \mathrm{sub}_{\lceil \frac{\delta}{2} \rceil}(\mathcal{T}')$ \cite[Proposition~2.15]{er}.

We now define the first higher Stasheff--Tamari order. Consider the cyclic polytope $C(\delta+2,\delta)$. Then any triangulation $\mathcal{T}$ of $C(\delta+2,\delta)$ determines a section $s_{\mathcal{T}}\colon C(\delta+2,\delta)\rightarrow C(\delta+2,\delta+1)$. But $C(\delta+2,\delta+1)$ is a simplex. It therefore has only one triangulation and only two sections: one corresponding to its upper facets and one corresponding to its lower facets. Hence the only two triangulations of $C(\delta+2,\delta)$ are the upper triangulation and the lower triangulation. For example, when $\delta=2$ the polytope $C(\delta+2,\delta)$ is a quadrilateral. This has two triangulations, corresponding to the two possible diagonals.

Let $\mathcal{T} \in \mathcal{S}(m,\delta)$. Suppose that there exists an increasing $(\delta + 2)$-tuple $H \in [m]^{\delta + 2}$ such that $\mathcal{T}$ restricts to the lower triangulation of $C(H,\delta)$. Let $\mathcal{T}'$ be the triangulation obtained by replacing the portion of $\mathcal{T}$ inside $C(H,\delta)$ with the upper triangulation of $C(H,\delta)$. We then say that $\mathcal{T}'$ is an \emph{increasing bistellar flip} of $\mathcal{T}$ and that $\mathcal{T}$ is a \emph{decreasing bistellar flip} of $\mathcal{T}'$. Figure~\ref{fig:flips} illustrates increasing bistellar flips for $\delta = 1, 2, 3$. Here the triangulations in the bottom row are the lower triangulations of $C(H, \delta)$ and the triangulations in the top row are the upper triangulations of $C(H, \delta)$. For $\delta = 1, 2$, we also illustrate the simplex $C(H, \delta + 1)$, with its two sections corresponding to the two different possible triangulations of $C(H, \delta)$.

\begin{figure}
\caption{Increasing bistellar flips}\label{fig:flips}
\[
\begin{tikzpicture}

\begin{scope}[scale=0.5,shift={(16,-2)}]

\begin{scope}[shift={(0,-6)}]


\coordinate(1a) at (-2.77,2.25);
\coordinate(2a) at (-1.77,0.75);
\coordinate(3a) at (-0.77,0.25);
\coordinate(4a) at (0.23,0.75);

\coordinate(1b) at (-1.75,2.5);
\coordinate(2b) at (-0.75,1);
\coordinate(4b) at (1.15,1);
\coordinate(5b) at (2.25,2.5);

\coordinate(2c) at (-0.25,0.25);
\coordinate(3c) at (0.75,-0.25);
\coordinate(4c) at (1.75,0.25);
\coordinate(5c) at (2.75,1.75);


\draw[fill=gray!30, draw=none] (1b) -- (2b) -- (5b) -- (1b);
\draw[fill=gray!30, draw=none] (2c) -- (5c) -- (3c) -- (2c);
\draw[fill=gray!30, draw=none] (3c) -- (4c) -- (5c) -- (3c);



\path[name path = line 1] (1a) -- (4a);


\path[name path = line 2] (1b) -- (2b);
\path[name path = line 3] (2b) -- (4b);


\path [name intersections={of = line 1 and line 2}];
\coordinate (a)  at (intersection-1);

\path [name intersections={of = line 1 and line 3}];
\coordinate (b)  at (intersection-1);


\draw[fill=red!30, draw=none] (1a) -- (a) -- (2b) -- (2a) -- (1a);
\draw[fill=red!30, draw=none] (2a) -- (2b) -- (b) -- (4a) -- (2a);
\draw[fill=red!30, draw=none] (2a) -- (4a) -- (3a) -- (2a);
\draw[fill=red!30, draw=none] (2b) -- (5b) -- (4b) -- (2b);


\draw (1a) -- (a);
\draw[dotted] (a) -- (b);
\draw (b) -- (4a);


\draw (1a) -- (2a);
\draw (2a) -- (3a);
\draw (3a) -- (4a);
\draw (2a) -- (4a);
\draw[dotted] (1a) -- (3a); 


\draw (1b) -- (5b);
\draw (1b) -- (2b);
\draw (2b) -- (4b);
\draw (4b) -- (5b);
\draw (2b) -- (5b);
\draw[dotted] (1b) -- (4b);


\draw (2c) -- (5c);
\draw (2c) -- (3c);
\draw (3c) -- (4c);
\draw (4c) -- (5c);
\draw (3c) -- (5c);
\draw[dotted] (2c) -- (4c);


\node at (1a) {$\bullet$};
\node at (2a) {$\bullet$};
\node at (3a) {$\bullet$};
\node at (4a) {$\bullet$};

\node at (1b) {$\bullet$};
\node at (2b) {$\bullet$};
\node at (4b) {$\bullet$};
\node at (5b) {$\bullet$};

\node at (2c) {$\bullet$};
\node at (3c) {$\bullet$};
\node at (4c) {$\bullet$};
\node at (5c) {$\bullet$};

\end{scope}

\begin{scope}[shift={(0,-2)}]

\node at (0,0) {4-simplex};

\end{scope}

\begin{scope}[shift={(0,0)}]


\coordinate(31f) at (-2.5,1.75);
\coordinate(32f) at (-1.5,0.2);
\coordinate(33f) at (-0.5,-0.25);
\coordinate(35f) at (1.5,1.75);

\coordinate(31b) at (-1.5,2.25);
\coordinate(33b) at (0.5,0.25);
\coordinate(34b) at (1.5,0.75);
\coordinate(35b) at (2.5,2.25);



\path[name path = line 2] (33f) -- (35f);
\path[name path = line 1] (35f) -- (31f);


\path[name path = line 3] (31b) -- (33b);


\path [name intersections={of = line 1 and line 3}];
\coordinate (a)  at (intersection-1);

\path [name intersections={of = line 2 and line 3}];
\coordinate (b)  at (intersection-1);


\draw[fill=gray!30, draw=none] (31f) -- (32f) -- (35f) -- (31f);
\draw[fill=gray!30, draw=none] (32f) -- (33f) -- (35f) -- (32f);
\draw[fill=gray!30, draw=none] (33b) -- (34b) -- (35b) -- (33b);


\draw[fill=green!30, draw=none] (31b) -- (a) -- (35f) -- (35b) -- (31b);
\draw[fill=green!30, draw=none] (33b) -- (b) -- (35f) -- (35b) -- (33b);



\draw (31b) -- (a);
\draw[dotted] (a) -- (b);
\draw (b) -- (33b);


\draw (31f) -- (32f);
\draw (32f) -- (33f);
\draw (33f) -- (35f);
\draw (35f) -- (31f);
\draw (32f) -- (35f);
\draw[dotted] (31f) -- (33f); 


\draw (33b) -- (34b);
\draw (34b) -- (35b);
\draw (35b) -- (31b);
\draw (33b) -- (35b);
\draw[dotted] (31b) -- (34b);


\node at (31f) {$\bullet$};
\node at (32f) {$\bullet$};
\node at (33f) {$\bullet$};
\node at (35f) {$\bullet$};
\node at (31b) {$\bullet$};
\node at (33b) {$\bullet$};
\node at (34b) {$\bullet$};
\node at (35b) {$\bullet$};

\end{scope}

\end{scope}

\begin{scope}[shift={(4.2,0)},scale=0.9]

\begin{scope}[shift={(0,0)}]


\coordinate (1) at (180:1);
\coordinate (2) at (240:1);
\coordinate (3) at (300:1);
\coordinate (4) at (0:1);


\draw (1) -- (2) -- (3) -- (4) -- (1);

\draw[green] (2) -- (4);


\node at (1) {$\bullet$};
\node at (2) {$\bullet$};
\node at (3) {$\bullet$};
\node at (4) {$\bullet$};

\end{scope}

\begin{scope}[shift={(0,-1.7)}]


\coordinate (1) at (180:1);
\coordinate (2) at (240:1);
\coordinate (3) at (300:1);
\coordinate (4) at (0:1);


\draw[draw=black,fill=gray!30] (1) -- (2) -- (3) -- (4) -- (1);

\draw[green] (2) -- (4);
\draw[red,dotted] (1) -- (3);


\node at (1) {$\bullet$};
\node at (2) {$\bullet$};
\node at (3) {$\bullet$};
\node at (4) {$\bullet$};

\end{scope}

\begin{scope}[shift={(0,-3.4)}]


\coordinate (1) at (180:1);
\coordinate (2) at (240:1);
\coordinate (3) at (300:1);
\coordinate (4) at (0:1);


\draw (1) -- (2) -- (3) -- (4) -- (1);

\draw[red] (1) -- (3);


\node at (1) {$\bullet$};
\node at (2) {$\bullet$};
\node at (3) {$\bullet$};
\node at (4) {$\bullet$};

\end{scope}

\end{scope}

\begin{scope}[shift={(0,0)}]

\begin{scope}[shift={(0,-0.25)}]


\coordinate(1) at (-1,0);
\coordinate(2) at (0,0);
\coordinate(3) at (1,0);


\draw[green] (1) -- (3);


\node at (1) {$\bullet$};
\node at (3) {$\bullet$};

\end{scope}

\begin{scope}[shift={(0,-1.5)}]


\coordinate(1) at (-1,0);
\coordinate(2) at (0,-1);
\coordinate(3) at (1,0);


\draw[green] (1) -- (3);
\draw[red] (1) -- (2) -- (3);


\node at (1) {$\bullet$};
\node at (2) {\color{red}$\bullet$};
\node at (3) {$\bullet$};

\end{scope}

\begin{scope}[shift={(0,-3.5)}]


\coordinate(1) at (-1,0);
\coordinate(2) at (0,0);
\coordinate(3) at (1,0);


\draw[red] (1) -- (3);


\node at (1) {$\bullet$};
\node at (2) {\color{red}$\bullet$};
\node at (3) {$\bullet$};

\end{scope}

\end{scope}

\end{tikzpicture}
\]
\end{figure}
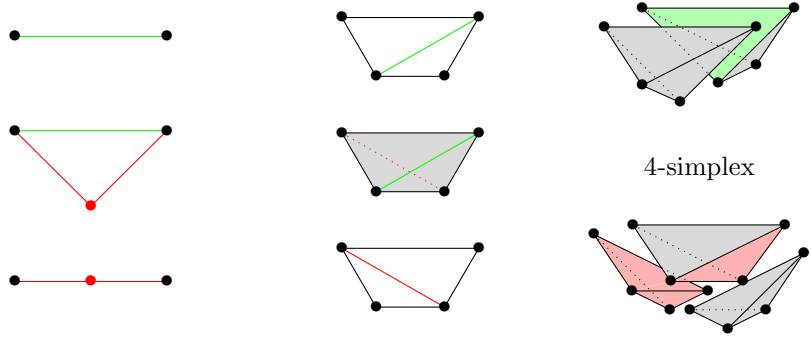

The covering relations of the \emph{first higher Stasheff--Tamari order} are that $\mathcal{T} \lessdot_{1} \mathcal{T}'$ if and only if $\mathcal{T}'$ is an increasing bistellar flip of $\mathcal{T}$. We write $\mathcal{S}_{1}(m,\delta)$ for the poset on $\mathcal{S}(m,\delta)$ this gives and $\leqslant_{1}$ for the partial order itself. The first higher Stasheff--Tamari order was originally introduced by Kapranov and Voevodsky in \cite[Definition~3.3]{kv-poly} as the ``higher Stasheff order'' using a slightly different definition. Thomas showed in \cite[Proposition~3.3]{thomas-bst} that the higher Stasheff order of Kapranov and Voevodsky was the same as the first higher Stasheff--Tamari order of Edelman and Reiner. General introductions to the higher Stasheff--Tamari orders can be found in \cite{rr} and \cite[Section~6.1]{lrs}.

\subsection{Higher Auslander--Reiten theory}\label{subsect-high-ar}

In this section let $\Lambda$ be a finite-dimen\-sional algebra over a field $K$. We denote by $\modules\Lambda$ the category of finite-dimensional right $\Lambda$-modules. Given a module $M \in \modules\Lambda$, $\additive M$ is the subcategory consisting of summands of direct sums of copies of $M$.

Given a subcategory $\mathcal{X} \subset \modules\Lambda$ and a map $f\colon X \rightarrow M$, where $X \in \mathcal{X}$ and $M \in \modules\Lambda$, we say that $f$ is a \emph{right $\mathcal{X}$-approximation} if for any $X' \in \mathcal{X}$, the sequence \[\Hom_{\Lambda}(X',X) \rightarrow \Hom_{\Lambda}(X',M)\rightarrow 0\] is exact, following \cite{as-preproj}. \emph{Left $\mathcal{X}$-approximations} are defined dually. The subcategory $\mathcal{X}$ is said to be \emph{contravariantly finite} if every $M \in \modules\Lambda$ admits a right $\mathcal{X}$-approximation, and \emph{covariantly finite} if every $M \in \modules\Lambda$ admits a left $\mathcal{X}$-approximation. If $\mathcal{X}$ is both contravariantly finite and covariantly finite, then $\mathcal{X}$ is \emph{functorially finite}.

Higher Auslander--Reiten theory was introduced by Iyama in \cite{iy-aus,iy-high-ar,iy-clus} as a higher-dimensional generalisation of classical Auslander--Reiten theory. For more background to the theory, see the papers \cite{jk-intro,jasso,gko,io,iy-revis}. The following subcategories provide the setting for higher Auslander--Reiten theory. A subcategory $\mathcal{M}$ of $\modules \Lambda$ is \emph{$d$-cluster-tilting} if it is functorially finite and
\begin{align*}
\mathcal{M} &= \{ X \in \modules\Lambda \mid \forall i \in [d-1],\, \forall M \in \mathcal{M},\, \Ext_{\Lambda}^{i}(X,M) = 0 \} \\
&= \{ X \in \modules\Lambda \mid \forall i \in [d-1],\, \forall M \in \mathcal{M},\, \Ext_{\Lambda}^{i}(M,X) = 0 \}.
\end{align*}
In the case $d=1$, the conditions should be interpreted as being trivial, so that $\modules\Lambda$ is the unique $1$-cluster-tilting subcategory of $\modules\Lambda$. If, for $M \in \modules\Lambda$, we have that $\additive M$ is a $d$-cluster-tilting subcategory, then we say that $M$ is a \emph{$d$-cluster-tilting module}.

We say that $\Lambda$ is \emph{weakly $d$-representation-finite} if there exists a $d$-cluster-tilting module in $\modules\Lambda$, following \cite[Definition~2.2]{io}. If, additionally, $\mathrm{gl.dim}\,\Lambda \leqslant d$, we say that $\Lambda$ is \emph{$d$-representation-finite $d$-hereditary}, following \cite[Definition~1.25]{jk-nak} and \cite{hio}. (These latter algebras are simply called ``$d$-representation-finite'' in \cite{io}.)

The canonical examples of $d$-representation-finite $d$-hereditary algebras are the higher Auslander algebras of linearly oriented $A_{n}$, introduced by Iyama in \cite{iy-clus}. The construction we give here is based on \cite[Construction~3.4]{ot}.

Following \cite{ot}, we denote the sets
\begin{align*}
\modset{m}{d} &:= \{(a_{0}, \dots , a_{d}) \in [m]^{d+1} \mid \forall i \in \{ 0, 1, \dots ,d-1 \}, a_{i+1}\geqslant a_{i}+2 \},\\
\nonconsec{m}{d} &:= \{(a_{0}, \dots , a_{d}) \in \modset{m}{d} \mid a_{d}+ 2 \leqslant a_{0} + m \}.
\end{align*}
Following \cite[Definition~2.2]{ot}, if $A, B \in [m]^{d + 1}$ are increasing $(d+1)$-tuples, then we say that $A$ \emph{intertwines} $B$, and write $A \wr B$, if \[a_{0}<b_{0}<a_{1}<b_{1}<\dots<a_{d}<b_{d}.\] If either $A \wr B$ or $B \wr A$, then we say that $A$ and $B$ are \emph{intertwining}. (That is, we use `are intertwining' to refer to the symmetric closure of the relation `intertwines'.)  A collection of increasing $(d + 1)$-tuples is called \emph{non-intertwining} if no pair of the elements are intertwining.

Let $Q^{(d,n)}$ be the quiver with vertices \[Q_{0}^{(d,n)} := \modset{n+2d-2}{d-1}\] and arrows \[Q_{1}^{(d,n)} := \{A \rightarrow A+e_{i} \mid A, A+e_{i} \in Q_{0}^{(d,n)}\},\] where \[e_{i}:= (0, \dots, 0, \overset{i}{1}, 0 , \dots , 0).\] We give examples of these quivers in Figure~\ref{fig:quiv_example}. We multiply arrows as if we were composing functions, so that $\xrightarrow{\alpha}\xrightarrow{\beta}\,=\beta\alpha$.

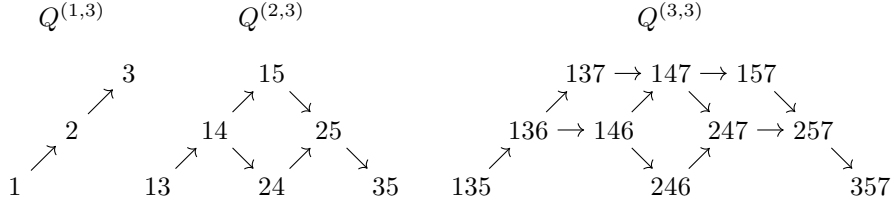
\begin{figure}
\caption{Examples of the quivers $Q^{(d, n)}$}\label{fig:quiv_example}
\[
\begin{tikzpicture}

\begin{scope}[scale=0.75]

\begin{scope}[shift={(-3,0)}]


\node(1) at (0,0) {1};
\node(2) at (1,1) {2};
\node(3) at (2,2) {3};

\draw[->] (1) -- (2);
\draw[->] (2) -- (3);


\node at (1,3) {$Q^{(1, 3)}$};

\begin{scope}[shift={(-0.5,0)}]


\node(13) at (3,0) {13};
\node(14) at (4,1) {14};
\node(15) at (5,2) {15};
\node(24) at (5,0) {24};
\node(25) at (6,1) {25};
\node(35) at (7,0) {35};

\draw[->] (13) -- (14);
\draw[->] (14) -- (15);
\draw[->] (14) -- (24);
\draw[->] (15) -- (25);
\draw[->] (24) -- (25);
\draw[->] (25) -- (35);


\node at (5,3) {$Q^{(2, 3)}$};

\end{scope}

\end{scope}

\begin{scope}[shift={(-3.5,0)}]


\node(135) at (8.5,0) {135};
\node(136) at (9.5,1) {136};
\node(137) at (10.5,2) {137};
\node(146) at (11,1) {146};
\node(147) at (12,2) {147};
\node(157) at (13.5,2) {157};
\node(246) at (12,0) {246};
\node(247) at (13,1) {247};
\node(257) at (14.5,1) {257};
\node(357) at (15.5,0) {357};

\draw[->] (135) -- (136);
\draw[->] (136) -- (137);
\draw[->] (136) -- (146);
\draw[->] (137) -- (147);
\draw[->] (146) -- (147);
\draw[->] (147) -- (157);
\draw[->] (146) -- (246);
\draw[->] (147) -- (247);
\draw[->] (157) -- (257);
\draw[->] (246) -- (247);
\draw[->] (247) -- (257);
\draw[->] (257) -- (357);


\node at (12,3) {$Q^{(3, 3)}$};

\end{scope}

\end{scope}

\end{tikzpicture}
\]
\end{figure}

Let $A_{n}^{d}$ be the quotient of the path algebra $KQ^{(d,n)}$ by the relations:
\begin{equation*}
A \rightarrow A+e_{i} \rightarrow A+e_{i}+e_{j} = 
\left\{
	\begin{array}{ll}
		A \rightarrow A+e_{j} \rightarrow A+e_{i}+e_{j} & \quad \text{if } A+e_{j} \in Q_{0}^{(d,n)} \\
		\hfil 0 & \quad \text{otherwise.}
	\end{array}
\right.
\end{equation*}

It is shown in \cite{iy-clus} that the algebra $A_{n}^{d}$ is $d$-representation-finite $d$-hereditary with unique basic $d$-cluster-tilting module $M^{(d,n)}$ and that \[ A_{n}^{d+1} \cong \mathrm{End}_{A_{n}^{d}}M^{(d,n)}.\]

Oppermann and Thomas give two different interpretations of triangulations of even-dimensional cyclic polytopes in the representation theory of $A_{n}^{d}$. In \cite[Section~3, Section 4]{ot} they give a bijection between triangulations and tilting modules; and in \cite[Section~5, Section~6]{ot}, they give a bijection between triangulations and cluster-tilting objects. Interesting results arise in both frameworks, so we consider both, although we shall need to modify the cluster-tilting framework from \cite{ot} in order to accommodate partial orders. For example, the tilting framework reveals connections with the orders on tilting modules studied in \cite{rs-simp,hu-potm}; and the cluster-tilting framework relates to maximal green sequences. We now explain the two different settings.

\subsubsection{Tilting}\label{sect-back-tilt}

Following \cite[Definition~2.3]{cps}, given a $\Lambda$-module $T$, we say that $T$ is a \emph{tilting} module if:
\begin{enumerate}
\item the projective dimension of $T$ is finite;
\item $\Ext_{\Lambda}^{i}(T,T)=0$ for all $i>0$;
\item there is an exact sequence $0 \rightarrow \Lambda \rightarrow T_{0} \rightarrow \dots \rightarrow T_{s} \rightarrow 0$ with each $T_{i} \in \additive T$.
\end{enumerate}

Given a $d$-representation-finite $d$-hereditary algebra $\Lambda$ with $d$-cluster-tilting module $M$ and tilting modules $T, T' \in \additive M$, we say that $T'$ is a \emph{left mutation} of $T$ if and only if $T=E \oplus X,$ $T'=E \oplus Y$ and there is an exact sequence \[0 \rightarrow X \rightarrow E_{1} \rightarrow \dots \rightarrow E_{d} \rightarrow Y \rightarrow 0\] where $X$ and $Y$ are indecomposable and $E_{i} \in \additive E$.

Given a tilting module $T \in \additive M^{(d,n)}$, we denote \[\prescript{\bot}{}T := \{ X \in \additive M^{(d,n)} \mid \Ext_{A_{n}^{d}}^{i}(X,T)=0, \: \forall i >0\}.\] The set $T^{\bot}$ is defined dually. Since $\mathrm{gl.dim}\,A_{n}^{d} \leqslant d$ and $\additive M^{(d,n)}$ is a $d$-cluster-tilting subcategory, we have that \[\prescript{\bot}{}T = \{ X \in \additive M^{(d,n)} \mid \Ext_{A_{n}^{d}}^{d}(X,T)=0\}.\]

By \cite[Theorem 1.1]{ot}, there are bijections between
\begin{itemize}
\item elements of $\modset{n+2d}{d}$,
\item $d$-simplices of $C(n+2d,2d)$ which do not lie in any lower facet, and
\item indecomposable direct summands of $M^{(d,n)}$.
\end{itemize}
These induce bijections between
\begin{itemize}
\item non-intertwining collections of $\binom{n+d-1}{d}$ increasing $(d + 1)$-tuples in $\modset{n+2d}{d}$,
\item triangulations of $C(n+2d,2d)$, and
\item basic tilting $A_{n}^{d}$-modules contained in $M^{(d,n)}$,
\end{itemize}
and also bijections between
\begin{itemize}
\item elements of $\nonconsec{n+2d}{d}$,
\item internal $d$-simplices of $C(n+2d,2d)$, and
\item indecomposable non-projective-injective direct summands of $M^{(d,n)}$.
\end{itemize}
Here an \emph{internal} simplex of $C(n + 2d, 2d)$ is one that does not lie in any facet of the polytope.

\begin{remark}
Hence, a triangulation of a $2d$-dimensional cyclic polytope can be understood as a collection of $d$-simplices, just as a triangulation of a convex polygon can be understood as a collection of arcs. While the two-dimensional case of this phenomenon is intuitive, the higher-dimensional case is less so. One can think about the higher-dimensional case as follows. Given a maximal-size collection of internal $d$-simplices whose vertex tuples do not intertwine, there is precisely one collection of $2d$-simplices which triangulate the cyclic polytope and which have the collection of internal $d$-simplices as their $d$-dimensional faces, along with the boundary $d$-simplices. For more details, see \cite[Lemma~2.15]{ot} and \cite{dey}.
\end{remark}

Given a triangulation $\mathcal{T} \in \mathcal{S}(n + 2d, 2d)$, we write $e(\mathcal{T})$ for the corresponding set of non-intertwining increasing $(d + 1)$-tuples in $\modset{n + 2d}{d}$, following \cite{ot}. Given an element $A \in \modset{n + 2d}{d}$, we use $M_{A}$ to refer to the corresponding indecomposable summand of $M^{(d,n)}$. The combinatorial bijection encodes homological properties of the category; namely, we have that $\Ext_{A_{n}^{d}}^{d}(M_{B}, M_{A}) \neq 0$ if and only if $A \wr B$.

\begin{example}\label{ex:back_tilt}
We first consider the example where $d = 1$ and $n = 3$. The bijection between $\modset{5}{1}$ and the indecomposables in $\additive M^{(1, 3)} = \modules A_{3}$ is shown in Figure~\ref{fig:mod_ar_quiv}. This bijection induces a bijection between triangulations of $C(5, 2)$ and basic tilting modules in $\modules A_{3}$, as shown in Figure~\ref{fig:mod_2d_triangs}. For example, the module $1$ is $M_{13}$, and so corresponds to the diagonal between vertex 1 and vertex 3 in Figure~\ref{fig:mod_2d_triangs}.

Now consider the example where $d = 2$ and $n = 3$. The algebra $A_{3}^{2}$ has the Auslander--Reiten quiver of $A_{3}$ as its quiver. To make the modules of this algebra easier to denote, we relabel the quiver \[
\begin{tikzpicture}

\begin{scope}[scale=0.75]

\node(13) at (-2,0) {1};
\node(14) at (-1,1) {2};
\node(15) at (0,2) {3};
\node(24) at (0,0) {4};
\node(25) at (1,1) {5};
\node(35) at (2,0) {6.};

\draw[->] (13) -- (14);
\draw[->] (14) -- (15);
\draw[->] (14) -- (24);
\draw[->] (15) -- (25);
\draw[->] (24) -- (25);
\draw[->] (25) -- (35);

\end{scope}

\end{tikzpicture}
\]
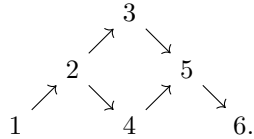
Figure~\ref{fig:2clus_ar_quiver} then shows the bijection between $\modset{7}{2}$ and the indecomposable modules of $\additive M^{(2, 3)}$. There are seven tilting modules in $\additive M^{(2, 3)}$, which correspond to the seven triangulations of $C(7, 4)$. This bijection is given in Table~\ref{tab:mod_4d_triangs}, where the triangulations $\mathcal{T}$ are described by their set $e(\mathcal{T})$ of $d$-simplices which do not lie in any lower facets of the polytope.
\end{example}

\begin{figure}
\caption{The Auslander--Reiten quiver of $\modules A_{3}$ and $\modset{5}{1}$.}\label{fig:mod_ar_quiv}
\[
\begin{tikzpicture}

\begin{scope}[scale=0.75,font=\tiny]

\begin{scope}[shift={(-3,0)}]


\node(1) at (0,0){$\tcs{1}$};
\node(21) at (1,1){$\tcs{2\\1}$};
\node(2) at (2,0){$\tcs{2}$};
\node(321) at (2,2){$\tcs{3\\2\\1}$};
\node(32) at (3,1){$\tcs{3\\2}$};
\node(3) at (4,0){$\tcs{3}$};


\draw[->] (1) -- (21);
\draw[->] (21) -- (2);
\draw[->] (2) -- (32);
\draw[->] (32) -- (3);
\draw[->] (21) -- (321);
\draw[->] (321) -- (32);

\end{scope}


\begin{scope}[shift={(3,0)}]


\node(1) at (0,0){$M_{13}$};
\node(21) at (1,1){$M_{14}$};
\node(2) at (2,0){$M_{24}$};
\node(321) at (2,2){$M_{15}$};
\node(32) at (3,1){$M_{25}$};
\node(3) at (4,0){$M_{35}$};


\draw[->] (1) -- (21);
\draw[->] (21) -- (2);
\draw[->] (2) -- (32);
\draw[->] (32) -- (3);
\draw[->] (21) -- (321);
\draw[->] (321) -- (32);

\end{scope}

\end{scope}

\end{tikzpicture}
\]
\end{figure}

\begin{figure}
\caption{Tilting modules in $\modules A_{3}$ and their corresponding triangulations.}\label{fig:mod_2d_triangs}
\[
\begin{tikzpicture}


\begin{scope}[shift={(-5,0)},scale=0.5,font=\tiny]


\coordinate(1) at (-2,2);
\coordinate(2) at (-1,0.5);
\coordinate(3) at (0,0);
\coordinate(4) at (1,0.5);
\coordinate(5) at (2,2);


\node(v1) at (1) {$\bullet$};
\node(v2) at (2) {$\bullet$};
\node(v3) at (3) {$\bullet$};
\node(v4) at (4) {$\bullet$};
\node(v5) at (5) {$\bullet$};


\node(mod2) at (0.5*-2 + 0.5*0, 0.5*2 + 0.5*0) {\color{red}$\tcs{1}$};
\node(mod12) at (0.5*-2 + 0.5*1, 0.5*2 + 0.5*0.5) {\color{red}$\tcs{2\\1}$};
\node(mod321) at ($(1)!0.5!(5)$) {\color{red} $\tcs{3\\2\\1}$};


\draw (1) -- (2) -- (3) -- (4) -- (5) -- (mod321) -- (1);
\draw[blue] (1) -- (mod2) -- (3);
\draw[blue] (1) -- (mod12) -- (4);

\end{scope}


\begin{scope}[shift={(-2.5,0)},scale=0.5,font=\tiny]


\coordinate(1) at (-2,2);
\coordinate(2) at (-1,0.5);
\coordinate(3) at (0,0);
\coordinate(4) at (1,0.5);
\coordinate(5) at (2,2);


\node(v1) at (1) {$\bullet$};
\node(v2) at (2) {$\bullet$};
\node(v3) at (3) {$\bullet$};
\node(v4) at (4) {$\bullet$};
\node(v5) at (5) {$\bullet$};


\node(mod1) at (0.5*-1 + 0.5*1, 0.5*0.5 + 0.5*0.5) {\color{red}$\tcs{2}$};
\node(mod12) at (0.5*-2 + 0.5*1, 0.5*2 + 0.5*0.5) {\color{red}$\tcs{2\\1}$};
\node(mod321) at ($(1)!0.5!(5)$) {\color{red} $\tcs{3\\2\\1}$};


\draw (1) -- (2) -- (3) -- (4) -- (5) -- (mod321) -- (1);
\draw[blue] (2) -- (mod1) -- (4);
\draw[blue] (1) -- (mod12) -- (4);

\end{scope}


\begin{scope}[shift={(0,0)},scale=0.5,font=\tiny]


\coordinate(1) at (-2,2);
\coordinate(2) at (-1,0.5);
\coordinate(3) at (0,0);
\coordinate(4) at (1,0.5);
\coordinate(5) at (2,2);


\node(v1) at (1) {$\bullet$};
\node(v2) at (2) {$\bullet$};
\node(v3) at (3) {$\bullet$};
\node(v4) at (4) {$\bullet$};
\node(v5) at (5) {$\bullet$};


\node(mod2) at (0.5*-2 + 0.5*0, 0.5*2 + 0.5*0) {\color{red}$\tcs{1}$};
\node(mod121) at (0.4*2 + 0.6*0, 0.4*2 + 0.6*0) {\color{red}$\tcs{3}$};
\node(mod321) at ($(1)!0.5!(5)$) {\color{red} $\tcs{3\\2\\1}$};


\draw (1) -- (2) -- (3) -- (4) -- (5) -- (mod321) -- (1);
\draw[blue] (1) -- (mod2) -- (3);
\draw[blue] (3) -- (mod121) -- (5);

\end{scope}


\begin{scope}[shift={(2.5,0)},scale=0.5,font=\tiny]


\coordinate(1) at (-2,2);
\coordinate(2) at (-1,0.5);
\coordinate(3) at (0,0);
\coordinate(4) at (1,0.5);
\coordinate(5) at (2,2);


\node(v1) at (1) {$\bullet$};
\node(v2) at (2) {$\bullet$};
\node(v3) at (3) {$\bullet$};
\node(v4) at (4) {$\bullet$};
\node(v5) at (5) {$\bullet$};


\node(mod1) at (0.5*-1 + 0.5*1, 0.5*0.5 + 0.5*0.5) {\color{red}$\tcs{2}$};
\node(mod21) at (0.5*-1 + 0.5*2, 0.5*0.5 + 0.5*2) {\color{red}$\tcs{3\\2}$};
\node(mod321) at ($(1)!0.5!(5)$) {\color{red} $\tcs{3\\2\\1}$};


\draw (1) -- (2) -- (3) -- (4) -- (5) -- (mod321) -- (1);
\draw[blue] (2) -- (mod1) -- (4);
\draw[blue] (2) -- (mod21) -- (5);

\end{scope}


\begin{scope}[shift={(5,0)},scale=0.5,font=\tiny]


\coordinate(1) at (-2,2);
\coordinate(2) at (-1,0.5);
\coordinate(3) at (0,0);
\coordinate(4) at (1,0.5);
\coordinate(5) at (2,2);


\node(v1) at (1) {$\bullet$};
\node(v2) at (2) {$\bullet$};
\node(v3) at (3) {$\bullet$};
\node(v4) at (4) {$\bullet$};
\node(v5) at (5) {$\bullet$};


\node(mod21) at (0.7*-1 + 0.3*2, 0.7*0.5 + 0.3*2) {\color{red}$\tcs{3\\2}$};
\node(mod121) at (0.4*2 + 0.6*0, 0.4*2 + 0.6*0) {\color{red}$\tcs{3}$};
\node(mod321) at ($(1)!0.5!(5)$) {\color{red} $\tcs{3\\2\\1}$};


\draw (1) -- (2) -- (3) -- (4) -- (5) -- (mod321) -- (1);
\draw[blue] (2) -- (mod21) -- (5);
\draw[blue] (3) -- (mod121) -- (5);

\end{scope}

\end{tikzpicture}
\]
\end{figure}
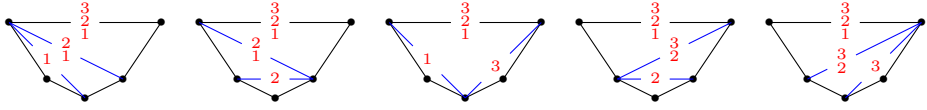

\begin{figure}
\caption{The Auslander--Reiten quiver of $\additive M^{(2,3)}$ and $\modset{7}{2}$.}\label{fig:2clus_ar_quiver}
\[
\begin{tikzpicture}


\begin{scope}[scale=0.75,font=\tiny]


\begin{scope}[shift={(-4,0)}]

\node(135) at (-3.5,0) {\tcs{1}};
\node(136) at (-2.5,1) {\tcs{2\\1}};
\node(137) at (-1.5,2) {\tcs{3\\2\\1}};
\node(146) at (-1,1) {\tcs{4\\2}};
\node(147) at (0,2) {\tcs{5\\4~3\\2}};
\node(157) at (1.5,2) {\tcs{6\\5\\3}};
\node(246) at (0,0) {\tcs{4}};
\node(247) at (1,1) {\tcs{5\\4}};
\node(257) at (2.5,1) {\tcs{6\\5}};
\node(357) at (3.5,0) {\tcs{6}};

\draw[->] (135) -- (136);
\draw[->] (136) -- (137);
\draw[->] (136) -- (146);
\draw[->] (137) -- (147);
\draw[->] (146) -- (147);
\draw[->] (147) -- (157);
\draw[->] (146) -- (246);
\draw[->] (147) -- (247);
\draw[->] (157) -- (257);
\draw[->] (246) -- (247);
\draw[->] (247) -- (257);
\draw[->] (257) -- (357);

\end{scope}



\begin{scope}[shift={(4,0)}]

\node(135) at (-3.5,0) {$M_{135}$};
\node(136) at (-2.5,1) {$M_{136}$};
\node(137) at (-1.5,2) {$M_{137}$};
\node(146) at (-1,1) {$M_{146}$};
\node(147) at (0,2) {$M_{147}$};
\node(157) at (1.5,2) {$M_{157}$};
\node(246) at (0,0) {$M_{246}$};
\node(247) at (1,1) {$M_{247}$};
\node(257) at (2.5,1) {$M_{257}$};
\node(357) at (3.5,0) {$M_{357}$};

\draw[->] (135) -- (136);
\draw[->] (136) -- (137);
\draw[->] (136) -- (146);
\draw[->] (137) -- (147);
\draw[->] (146) -- (147);
\draw[->] (147) -- (157);
\draw[->] (146) -- (246);
\draw[->] (147) -- (247);
\draw[->] (157) -- (257);
\draw[->] (246) -- (247);
\draw[->] (247) -- (257);
\draw[->] (257) -- (357);

\end{scope}


\end{scope}

\end{tikzpicture}
\]
\end{figure}

\begin{table}
\caption{Tilting modules in $\additive M^{(2, 3)}$ and their corresponding triangulations.}\label{tab:mod_4d_triangs}
\begin{tabularx}{\textwidth}{
  >{\centering\arraybackslash}X 
  | >{\centering\arraybackslash}X
  | >{\centering\arraybackslash}X}
Tilting module & Triangulation & $d$-simplices not in lower facets\\
\hline
$\tcs{1} \oplus \tcs{2\\1} \oplus \tcs{4\\2} \oplus \tcs{3\\2\\1} \oplus \tcs{5\\4~3\\2} \oplus \tcs{6\\5\\3}$ & $\{|12345|, |12356|, |13456|,$ $|12367|, |13467|, |14567|\}$ & $\{135, 136, 146,$ $137, 147, 157\}$\\
$\tcs{4} \oplus \tcs{2\\1} \oplus \tcs{4\\2} \oplus \tcs{3\\2\\1} \oplus \tcs{5\\4~3\\2} \oplus \tcs{6\\5\\3}$ & $\{|12346|, |12456|, |23456|,$ $|12367|, |13467|, |14567|\}$ & $\{246, 136, 146,$ $137, 147, 157\}$\\
$\tcs{4} \oplus \tcs{5\\4} \oplus \tcs{4\\2} \oplus \tcs{3\\2\\1} \oplus \tcs{5\\4~3\\2} \oplus \tcs{6\\5\\3}$ & $\{|12347|, |12456|, |23456|,$ $|12467|, |23467|, |14567|\}$ & $\{246, 247, 146,$ $137, 147, 157\}$\\
$\tcs{4} \oplus \tcs{5\\4} \oplus \tcs{6\\5} \oplus \tcs{3\\2\\1} \oplus \tcs{5\\4~3\\2} \oplus \tcs{6\\5\\3}$ & $\{|12347|, |12457|, |23456|,$ $|12567|, |23467|, |24567|\}$ & $\{246, 247, 257,$ $137, 147, 157\}$\\
$\tcs{1} \oplus \tcs{2\\1} \oplus \tcs{6} \oplus \tcs{3\\2\\1} \oplus \tcs{5\\4~3\\2} \oplus \tcs{6\\5\\3}$ & $\{|12345|, |12356|, |13457|,$ $|12367|, |13567|, |34567|\}$ & $\{135, 136, 357,$ $137, 147, 157\}$\\
$\tcs{1} \oplus \tcs{6\\5} \oplus \tcs{6} \oplus \tcs{3\\2\\1} \oplus \tcs{5\\4~3\\2} \oplus \tcs{6\\5\\3}$ & $\{|12345|, |12357|, |13457|,$ $|12567|, |23567|, |34567|\}$ & $\{135, 257, 357,$ $137, 147, 157\}$\\
$\tcs{6} \oplus \tcs{5\\4} \oplus \tcs{6\\5} \oplus \tcs{3\\2\\1} \oplus \tcs{5\\4~3\\2} \oplus \tcs{6\\5\\3}$ & $\{|12347|, |12457|, |23457|,$ $|12567|, |23567|, |34567|\}$ & $\{357, 247, 257,$ $137, 147, 157\}$
\end{tabularx}
\end{table}

\subsubsection{Cluster-tilting}\label{sect-back-clus-tilt}

Let $\Lambda$ be a $d$-representation-finite $d$-hereditary algebra with $d$-cluster-tilting module $M$. Let $\mathcal{D}_{\Lambda}:=D^{b}(\modules\Lambda)$ be the bounded derived category of finitely generated $\Lambda$-modules. We denote the shift functor in the derived category by $[1]$ and its $d$-th power by $[d]:=[1]^{d}$. Let \[\mathcal{U}_{\Lambda}:=\additive\{M[id] \in \mathcal{D}_{\Lambda} \mid i \in \mathbb{Z}\}\] be a subcategory of $\mathcal{D}_{\Lambda}$. This is a $(d+2)$-angulated category in the sense of \cite{gko}. Consider the subcategory $\clus_{\Lambda}:=\additive(M \oplus \Lambda[d])$ of $\mathcal{U}_{\Lambda}$. We say that $T \in \clus_{\Lambda}$ is a \emph{(basic) cluster-tilting object} if $\Hom_{\mathcal{U}_{\Lambda}}(T, T[d]) = 0$ and if $T$ has $m$ indecomposable summands which are pairwise non-isomorphic, where $m$ is the number of indecomposable summands of $\Lambda$ as a $\Lambda$-module. Note that both $\Lambda$ and $\Lambda[d]$ are cluster-tilting objects. 

Our set-up for cluster-tilting is different from that of \cite{ot} because we wish to consider orders on triangulations and cluster-tilting objects. Hence we consider $\clus_{\Lambda}$ rather than the higher cluster category $\mathcal{O}_{A_{n}^{d}}$ they define, which is $2d$-Calabi--Yau. We retain the term ``cluster-tilting'' to refer to the relevant objects, following, for example, \cite{hi-no-gap}.

Given two cluster-tilting objects $T,T' \in \clus_{\Lambda}$, we say that $T'$ is a \emph{left mutation} of $T$ if $T=E \oplus X$, $T'=E \oplus Y$ and there exists a $(d+2)$-angle \[X \rightarrow E_{1} \rightarrow \dots \rightarrow E_{d} \rightarrow Y \rightarrow X[d].\]

Similarly to before, given a cluster-tilting object $T \in \clus_{A_{n}^{d}}$, we denote 
\begin{align*}
\prescript{\bot}{}T &:= \{ X \in \clus_{A_{n}^{d}} \mid \Hom_{\mathcal{D}_{A_{n}^{d}}}(X,T[i])=0, \: \forall i >0\}\\
&=\{ X \in \clus_{A_{n}^{d}} \mid \Hom_{\mathcal{D}_{A_{n}^{d}}}(X,T[d])=0\}.
\end{align*}

By \cite[Theorem 1.1, Theorem 1.2]{ot} there are bijections between:
\begin{itemize}
\item elements of $\nonconsec{n+2d+1}{d}$,
\item internal $d$-simplices of $C(n+2d+1,2d)$,
\item indecomposable objects in $\clus_{A_{n}^{d}}$, and
\item indecomposable non-projective-injective direct summands of $M^{(d,n+1)}$.
\end{itemize}
These induce bijections between:
\begin{itemize}
\item non-intertwining sets of $\binom{n+d-1}{d}$ increasing $(d + 1)$-tuples in $\nonconsec{n+2d+1}{d}$,
\item triangulations of $C(n+2d+1,2d)$,
\item basic cluster-tilting objects in $\clus_{A_{n}^{d}}$, and
\item basic tilting $A_{n+1}^{d}$-modules contained in $M^{(d,n+1)}$.
\end{itemize}
As before, the combinatorics encodes extensions between objects, in that we have $\Hom_{\mathcal{U}_{A_{n}^{d}}}(M_{B}, M_{A}[d]) \neq 0$ if and only if $A \wr B$.

\begin{remark}
The reason that the two different frameworks behave so similarly is that the quotient of $\additive M^{(d, n + 1)}$ by the projective-injectives is equivalent to $\clus_{A_{n}^{d}}$ as a $(d + 2)$-exangulated category \cite{np,hln}.
\end{remark}

\begin{example}
We first consider the example where $d = 1$ and $n = 2$. The bijection between $\nonconsec{5}{1}$ and the indecomposables in $\clus_{A_{2}}$ is shown in Figure~\ref{fig:clus_ar_quiv}. This bijection induces a bijection between triangulations of $C(5, 2)$ and basic cluster-tilting objects in $\clus_{A_{2}}$, as shown in Figure~\ref{fig:2d_triangs}.

Now consider the example where $d = 2$ and $n = 2$. The algebra $A_{2}^{2}$ has the Auslander--Reiten quiver of $A_{2}$ as its quiver. To make the modules of this algebra easier to denote, we relabel the quiver \[
\begin{tikzpicture}

\begin{scope}[scale=0.75]

\node(13) at (-2,0) {1};
\node(14) at (-1,1) {2};
\node(24) at (0,0) {3.};

\draw[->] (13) -- (14);
\draw[->] (14) -- (24);

\end{scope}

\end{tikzpicture}
\]
Figure~\ref{fig:2clus_cat} then shows the bijection between $\nonconsec{7}{2}$ and the indecomposables of $\clus_{A_{2}^{2}}$. There are seven cluster-tilting objects in $\clus_{A_{2}^{2}}$, which correspond to the seven triangulations of $C(7, 4)$. This bijection is given in Table~\ref{tab:clus_4d_triangs}, where the triangulations are described by their set of internal $d$-simplices.

Compare these examples with Example~\ref{ex:back_tilt}. The difference is that in the cluster-tilting case, the indecomposable objects are now in bijection with the internal simplices of the triangulation. We no longer have indecomposable objects which correspond to simplices in the upper facets of the cyclic polytope. In the tilting case, simplices in the upper facets of the cyclic polytope correspond to projective-injectives.
\end{example}

\begin{figure}
\caption{The Auslander--Reiten quiver of $\clus_{A_{2}}$ and $\nonconsec{5}{1}$.}\label{fig:clus_ar_quiv}
\[
\begin{tikzpicture}

\begin{scope}[scale=0.75,font=\tiny]

\begin{scope}[shift={(-3,0)}]


\node(1) at (0,0){$\tcs{1}$};
\node(21) at (1,1){$\tcs{2\\1}$};
\node(2) at (2,0){$\tcs{2}$};
\node(32) at (3,1){$\tcs{1}[1]$};
\node(3) at (4,0){$\tcs{2\\1}[1]$};


\draw[->] (1) -- (21);
\draw[->] (21) -- (2);
\draw[->] (2) -- (32);
\draw[->] (32) -- (3);

\end{scope}


\begin{scope}[shift={(3,0)}]


\node(1) at (0,0){$M_{13}$};
\node(21) at (1,1){$M_{14}$};
\node(2) at (2,0){$M_{24}$};
\node(32) at (3,1){$M_{25}$};
\node(3) at (4,0){$M_{35}$};


\draw[->] (1) -- (21);
\draw[->] (21) -- (2);
\draw[->] (2) -- (32);
\draw[->] (32) -- (3);

\end{scope}

\end{scope}

\end{tikzpicture}
\]
\end{figure}

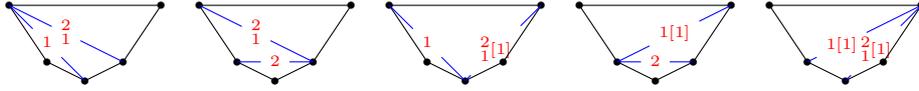
\begin{figure}
\caption{Cluster-tilting objects in $\clus_{A_{2}}$ and their corresponding triangulations.}\label{fig:2d_triangs}
\[
\begin{tikzpicture}


\begin{scope}[shift={(-5,0)},scale=0.5,font=\tiny]


\coordinate(1) at (-2,2);
\coordinate(2) at (-1,0.5);
\coordinate(3) at (0,0);
\coordinate(4) at (1,0.5);
\coordinate(5) at (2,2);


\node(v1) at (1) {$\bullet$};
\node(v2) at (2) {$\bullet$};
\node(v3) at (3) {$\bullet$};
\node(v4) at (4) {$\bullet$};
\node(v5) at (5) {$\bullet$};


\node(mod2) at (0.5*-2 + 0.5*0, 0.5*2 + 0.5*0) {\color{red}$\tcs{1}$};
\node(mod12) at (0.5*-2 + 0.5*1, 0.5*2 + 0.5*0.5) {\color{red}$\tcs{2\\1}$};


\draw (1) -- (2) -- (3) -- (4) -- (5) -- (1);
\draw[blue] (1) -- (mod2) -- (3);
\draw[blue] (1) -- (mod12) -- (4);

\end{scope}



\begin{scope}[shift={(-2.5,0)},scale=0.5,font=\tiny]


\coordinate(1) at (-2,2);
\coordinate(2) at (-1,0.5);
\coordinate(3) at (0,0);
\coordinate(4) at (1,0.5);
\coordinate(5) at (2,2);


\node(v1) at (1) {$\bullet$};
\node(v2) at (2) {$\bullet$};
\node(v3) at (3) {$\bullet$};
\node(v4) at (4) {$\bullet$};
\node(v5) at (5) {$\bullet$};


\node(mod1) at (0.5*-1 + 0.5*1, 0.5*0.5 + 0.5*0.5) {\color{red}$\tcs{2}$};
\node(mod12) at (0.5*-2 + 0.5*1, 0.5*2 + 0.5*0.5) {\color{red}$\tcs{2\\1}$};


\draw (1) -- (2) -- (3) -- (4) -- (5) -- (1);
\draw[blue] (2) -- (mod1) -- (4);
\draw[blue] (1) -- (mod12) -- (4);

\end{scope}


\begin{scope}[shift={(0,0)},scale=0.5,font=\tiny]


\coordinate(1) at (-2,2);
\coordinate(2) at (-1,0.5);
\coordinate(3) at (0,0);
\coordinate(4) at (1,0.5);
\coordinate(5) at (2,2);


\node(v1) at (1) {$\bullet$};
\node(v2) at (2) {$\bullet$};
\node(v3) at (3) {$\bullet$};
\node(v4) at (4) {$\bullet$};
\node(v5) at (5) {$\bullet$};


\node(mod2) at (0.5*-2 + 0.5*0, 0.5*2 + 0.5*0) {\color{red}$\tcs{1}$};
\node(mod121) at (0.4*2 + 0.6*0, 0.4*2 + 0.6*0) {\color{red}$\tcs{2\\1}[1]$};


\draw (1) -- (2) -- (3) -- (4) -- (5) -- (1);
\draw[blue] (1) -- (mod2) -- (3);
\draw[blue] (3) -- (mod121) -- (5);

\end{scope}


\begin{scope}[shift={(2.5,0)},scale=0.5,font=\tiny]


\coordinate(1) at (-2,2);
\coordinate(2) at (-1,0.5);
\coordinate(3) at (0,0);
\coordinate(4) at (1,0.5);
\coordinate(5) at (2,2);


\node(v1) at (1) {$\bullet$};
\node(v2) at (2) {$\bullet$};
\node(v3) at (3) {$\bullet$};
\node(v4) at (4) {$\bullet$};
\node(v5) at (5) {$\bullet$};


\node(mod1) at (0.5*-1 + 0.5*1, 0.5*0.5 + 0.5*0.5) {\color{red}$\tcs{2}$};
\node(mod21) at (0.5*-1 + 0.5*2, 0.5*0.5 + 0.5*2) {\color{red}$\tcs{1}[1]$};


\draw (1) -- (2) -- (3) -- (4) -- (5) -- (1);
\draw[blue] (2) -- (mod1) -- (4);
\draw[blue] (2) -- (mod21) -- (5);

\end{scope}


\begin{scope}[shift={(5,0)},scale=0.5,font=\tiny]


\coordinate(1) at (-2,2);
\coordinate(2) at (-1,0.5);
\coordinate(3) at (0,0);
\coordinate(4) at (1,0.5);
\coordinate(5) at (2,2);


\node(v1) at (1) {$\bullet$};
\node(v2) at (2) {$\bullet$};
\node(v3) at (3) {$\bullet$};
\node(v4) at (4) {$\bullet$};
\node(v5) at (5) {$\bullet$};


\node(mod21) at (0.7*-1 + 0.3*2, 0.7*0.5 + 0.3*2) {\color{red}$\tcs{1}[1]$};
\node(mod121) at (0.4*2 + 0.6*0, 0.4*2 + 0.6*0) {\color{red}$\tcs{2\\1}[1]$};


\draw (1) -- (2) -- (3) -- (4) -- (5) -- (1);
\draw[blue] (2) -- (mod21) -- (5);
\draw[blue] (3) -- (mod121) -- (5);

\end{scope}

\end{tikzpicture}
\]
\end{figure}

\begin{figure}
\caption{The Auslander--Reiten quiver of $\clus_{A_{2}^{2}}$ and $\nonconsec{7}{2}$.}\label{fig:2clus_cat}
\[
\begin{tikzpicture}


\begin{scope}[scale=0.75,font=\tiny]


\begin{scope}[shift={(-4,0)}]

\node(135) at (-3.5,0) {\tcs{1}};
\node(136) at (-2.5,1) {\tcs{2\\1}};
\node(146) at (-1,1) {\tcs{3\\2}};
\node(246) at (0,0) {\tcs{3}};
\node(247) at (1,1) {\tcs{1}[2]};
\node(257) at (2.5,1) {\tcs{2\\1}[2]};
\node(357) at (3.5,0) {\tcs{3\\2}[2]};

\draw[->] (135) -- (136);
\draw[->] (136) -- (146);
\draw[->] (146) -- (246);
\draw[->] (246) -- (247);
\draw[->] (247) -- (257);
\draw[->] (257) -- (357);

\end{scope}



\begin{scope}[shift={(4,0)}]

\node(135) at (-3.5,0) {$M_{135}$};
\node(136) at (-2.5,1) {$M_{136}$};
\node(146) at (-1,1) {$M_{146}$};
\node(246) at (0,0) {$M_{246}$};
\node(247) at (1,1) {$M_{247}$};
\node(257) at (2.5,1) {$M_{257}$};
\node(357) at (3.5,0) {$M_{357}$};

\draw[->] (135) -- (136);
\draw[->] (136) -- (146);
\draw[->] (146) -- (246);
\draw[->] (246) -- (247);
\draw[->] (247) -- (257);
\draw[->] (257) -- (357);

\end{scope}


\end{scope}

\end{tikzpicture}
\]
\end{figure}

\begin{table}
\caption{Cluster-tilting objects in $\clus_{A_{2}^{2}}$ and their corresponding triangulations.}\label{tab:clus_4d_triangs}
\begin{tabularx}{\textwidth}{
  >{\centering\arraybackslash}X 
  | >{\centering\arraybackslash}X
  | >{\centering\arraybackslash}X}
Cluster-tilting object & Triangulation & Internal $d$-simplices\\
\hline
$\tcs{1} \oplus \tcs{2\\1} \oplus \tcs{3\\2}$ & $\{|12345|, |12356|, |13456|,$ $|12367|, |13467|, |14567|\}$ & $\{135, 136, 146\}$\\
$\tcs{3} \oplus \tcs{2\\1} \oplus \tcs{3\\2}$ & $\{|12346|, |12456|, |23456|,$ $|12367|, |13467|, |14567|\}$ & $\{246, 136, 146\}$\\
$\tcs{3} \oplus \tcs{1}[2] \oplus \tcs{3\\2}$ & $\{|12347|, |12456|, |23456|,$ $|12467|, |23467|, |14567|\}$ & $\{246, 247, 146\}$\\
$\tcs{3} \oplus \tcs{1}[2] \oplus \tcs{2\\1}[2]$ & $\{|12347|, |12457|, |23456|,$ $|12567|, |23467|, |24567|\}$ & $\{246, 247, 257\}$\\
$\tcs{1} \oplus \tcs{2\\1} \oplus \tcs{3\\2}[2]$ & $\{|12345|, |12356|, |13457|,$ $|12367|, |13567|, |34567|\}$ & $\{135, 136, 357\}$\\
$\tcs{1} \oplus \tcs{2\\1}[2] \oplus \tcs{3\\2}[2]$ & $\{|12345|, |12357|, |13457|,$ $|12567|, |23567|, |34567|\}$ & $\{135, 257, 357\}$\\
$\tcs{1}[2] \oplus \tcs{2\\1}[2] \oplus \tcs{3\\2}[2]$ & $\{|12347|, |12457|, |23457|,$ $|12567|, |23567|, |34567|\}$ & $\{357, 247, 257\}$
\end{tabularx}
\end{table}

\section{The higher Stasheff--Tamari orders in even dimensions}\label{sect:even}

We now show how the higher Stasheff--Tamari orders may be interpreted algebraically in even dimensions. In this section we choose to work in the tilting framework.

\subsection{First order}

Our first result shows that increasing bistellar flips of triangulations correspond to left mutations of tilting modules.

\begin{theorem}\label{thm:even:first}
Let $\mathcal{T}, \mathcal{T}' \in \mathcal{S}(n+2d,2d)$ with corresponding tilting $A_{n}^{d}$-modules $T$ and $T'$. Then $\mathcal{T} \lessdot_{1} \mathcal{T}'$ if and only if $T'$ is a left mutation of $T$.
\end{theorem}
\begin{proof}
We know from \cite[Theorem 4.4]{ot} that $T$ and $T'$ are related by a mutation if and only if $\mathcal{T}$ and $\mathcal{T}'$ are related by a bistellar flip. Hence, what we need to show is that the mutation is a left mutation if and only if the bistellar flip is increasing. Let $T = E \oplus M_{A}$ and $T' = E \oplus M_{B}$. This means that the $(2d + 1)$-simplex inducing the increasing bistellar flip is $|A \cup B|$, and that $e(\mathcal{T}') = (e(\mathcal{T})\setminus \{A\}) \cup \{B\}$. Then $T'$ is a left mutation of $T$ if and only if $\Ext_{A_{n}^{d}}^{d}(M_{B}, M_{A}) \neq 0$, and we know this is the case if and only if $A \wr B$. Finally, we have that $A \wr B$ if and only if the bistellar flip is increasing, since one can show that, in this case, $|A|$ is the intersection of the lower facets of $|A \cup B|$ and $|B|$ is the intersection of its upper facets.
\end{proof}

\begin{example}\label{ex:even:first}
We illustrate this theorem with some examples, following on from Example~\ref{ex:back_tilt}. We start with the example where $n = 3$ and $d = 1$. If we take the triangulations given by $e(\mathcal{T}) = \{13, 14, 15\}$ and $e(\mathcal{T}') = \{24, 14, 15\}$, then $\mathcal{T}'$ is an increasing bistellar flip of $\mathcal{T}$ via the 3-simplex $|1234|$. The intersection of the lower facets of this simplex is $|13|$ and the intersection of its upper facets is $|24|$. The corresponding tilting modules are \[T = \tcs{1} \oplus \tcs{2\\1} \oplus \tcs{3\\2\\1} \quad \text{ and } \quad T' = \tcs{2} \oplus \tcs{2\\1} \oplus \tcs{3\\2\\1},\] which are related by the exchange sequence \[0 \to \tcs{1} \to \tcs{2\\1} \to \tcs{2} \to 0,\] so that $T'$ is a left mutation of $T$, as given by Theorem~\ref{thm:even:first}.

Now consider the example where $n = 3$ and $d = 2$. Take the triangulations given by $e(\mathcal{T}) = \{135, 136, 146, 137, 147, 157\}$ and $e(\mathcal{T}') = \{246, 136, 146, 137, 147, 157\}$. Here $\mathcal{T}'$ is an increasing bistellar flip of $\mathcal{T}$ via the 5-simplex $|123456|$. The intersection of the lower facets of this simplex is $|135|$ and the intersection of its upper facets is $|246|$. The corresponding tilting modules are then \[T = \tcs{1} \oplus \tcs{2\\1} \oplus \tcs{4\\2} \oplus \tcs{3\\2\\1} \oplus \tcs{5\\4~3\\2} \oplus \tcs{6\\5\\3} \quad \text{ and } \quad T' = \tcs{4} \oplus \tcs{2\\1} \oplus \tcs{4\\2} \oplus \tcs{3\\2\\1} \oplus \tcs{5\\4~3\\2} \oplus \tcs{6\\5\\3},\] which are related by the exchange sequence \[0 \to \tcs{1} \to \tcs{2\\1} \to \tcs{4\\2} \to \tcs{4} \to 0,\] so that $T'$ is again a left mutation of $T$, as Theorem~\ref{thm:even:first} dictates.
\end{example}

\subsection{Second order}\label{even-second}

We now show how the second higher Stasheff--Tamari order may be interpreted algebraically. We use the definition of the second order in terms of submersion, which we show has the following algebraic interpretation.

\begin{proposition}\label{prop:sub<->perp}
Let $\mathcal{T} \in \mathcal{S}(n+2d,2d)$ be a triangulation with corresponding tilting $A_{n}^{d}$-module $T$. Let $|A|$ be an internal $d$-simplex in $C(n+2d,2d)$ with corresponding indecomposable $A_{n}^{d}$-module $M_{A}$. Then $|A|$ is submerged by $\mathcal{T}$ if and only if $M_{A} \in \prescript{\bot}{}{T}$.
\end{proposition}
\begin{proof}[Sketch]
We have that $M_{A} \in \prescript{\bot}{}{T}$ if and only if for all indecomposable summands $M_{B}$ of $T$ we have $\Ext_{A_{n}^{d}}(M_{A}, M_{B}) = 0$, which is the case if and only if we do not have $B \wr A$. Hence we need to show that $|A|$ is submerged by $\mathcal{T}$ if and only if there is no $B \in e(\mathcal{T})$ such that $B \wr A$. If there is a $B \in e(\mathcal{T})$ such that $B \wr A$, then $|A|_{2d + 1}$ is the intersection of the upper facets of $|A \cup B|_{2d + 1}$ and $|B|_{2d + 1}$ is the intersection of its lower facets. Hence $|A|$ cannot be submerged by $\mathcal{T}$, since $|A|_{2d + 1}$ lies above $|B|_{2d + 1}$, and $B \in e(\mathcal{T})$. Conversely, if $|A|$ is not submerged by $\mathcal{T}$, then $|A|_{2d+1}$ either lies below $s_{\mathcal{T}}(C(n + 2d, 2d))$ or intersects $s_{\mathcal{T}}(C(n + 2d, 2d))$. In either case one can find a $B \in e(\mathcal{T})$ such that $B \wr A$.
\end{proof}

By applying Proposition~\ref{prop:sub<->perp} and using the interpretation of the second higher Stasheff--Tamari order in terms of submersion sets we obtain the following theorem.

\begin{theorem}\label{thm:even:second}
Let $\mathcal{T}, \mathcal{T}' \in \mathcal{S}(n+2d,2d)$ with corresponding tilting $A_{n}^{d}$-modules $T$ and $T'$. Then $\mathcal{T} \leqslant_{2} \mathcal{T}'$ if and only if $\prescript{\bot}{}T \subseteq \prescript{\bot}{}T'$.
\end{theorem}

\begin{remark}
Theorems~\ref{thm:even:first} and \ref{thm:even:second} open the door for a proof of the Edelman--Reiner conjecture using the techniques of homological algebra. The orders in these theorems are higher-dimensional versions of the orders on tilting modules introduced in \cite{rs-simp} for general finite-dimensional algebras over algebraically closed fields. They were shown to have the same Hasse diagram in \cite{hu-potm} for arbitrary Artin algebras, which implies that the orders on tilting modules for $d = 1$ are equal for a representation-finite algebra. A higher-dimensional version of this result would give an algebraic proof of the equivalence of the higher Stasheff--Tamari orders in even dimensions. Indeed, since \cite{hu-potm} shows that the two orders on tilting modules have the same Hasse diagram for $d = 1$, this provides a neat algebraic proof of the result that the two higher Stasheff--Tamari orders are equal in dimension two \cite[Theorem 3.8]{er}.
\end{remark}

\begin{example}\label{ex-even-second}
We illustrate Theorem~\ref{thm:even:second} with examples, following on from Example~\ref{ex:back_tilt}. We first consider the case where $n = 3$ and $d = 1$. Consider the triangulations given by $\mathcal{T} = \{13, 14, 15\}$ and $\mathcal{T}' = \{24, 25, 15\}$. The corresponding tilting modules are \[T = \tcs{1} \oplus \tcs{2\\1} \oplus \tcs{3\\2\\1} \quad \text{ and } \quad T' = \tcs{2} \oplus \tcs{3\\2} \oplus \tcs{3\\2\\1},\] whose orthogonal categories are
\begin{align*}
\orth{T} &= \additive \left\{\tcs{1},\, \tcs{2\\1},\, \tcs{3\\2\\1}\right\}\\
\orth{T'} &= \additive \left\{\tcs{1},\, \tcs{2\\1},\, \tcs{3\\2\\1},\, \tcs{2},\, \tcs{3\\2}\right\}.
\end{align*}
Since $\orth{T} \subseteq \orth{T'}$, we have that $\mathcal{T} \leqslant_{2} \mathcal{T}'$ by Theorem~\ref{thm:even:second}.

Now consider the example where $n = 3$ and $d = 2$. Take the triangulations given by $\mathcal{T} = \{135, 136, 146, 137, 147, 157\}$ and $\mathcal{T}' = \{135, 257, 357, 137, 147, 157\}$. The corresponding tilting modules are then \[T = \tcs{1} \oplus \tcs{2\\1} \oplus \tcs{4\\2} \oplus \tcs{3\\2\\1} \oplus \tcs{5\\4~3\\2} \oplus \tcs{6\\5\\3} \quad \text{ and } \quad T' = \tcs{1} \oplus \tcs{6\\5} \oplus \tcs{6} \oplus \tcs{3\\2\\1} \oplus \tcs{5\\4~3\\2} \oplus \tcs{6\\5\\3},\] whose orthogonal categories are
\begin{align*}
\orth{T} &= \additive \left\{\tcs{1},\, \tcs{2\\1},\, \tcs{4\\2},\, \tcs{3\\2\\1},\, \tcs{5\\4~3\\2},\, \tcs{6\\5\\3}\right\},\\
\orth{T'} &= \additive \left\{\tcs{1},\, \tcs{2\\1},\, \tcs{4\\2},\, \tcs{3\\2\\1},\, \tcs{5\\4~3\\2},\, \tcs{6\\5\\3},\, \tcs{6\\5},\, \tcs{6}\right\}.
\end{align*}
Since $\orth{T} \subseteq \orth{T'}$, we have that $\mathcal{T} \leqslant_{2} \mathcal{T}'$ by Theorem~\ref{thm:even:second}.
\end{example}

Naturally, analogues of Theorem~\ref{thm:even:first} and Theorem~\ref{thm:even:second} also hold in the cluster-tilting framework.

\section{The higher Stasheff--Tamari orders in odd dimensions}\label{sect:odd}

We now give combinatorial and algebraic interpretations of the higher Stasheff--Tamari orders on triangulations of odd-dimensional cyclic polytopes. To obtain the algebraic interpretations, we first show how triangulations of odd-dimensional cyclic polytopes arise in the representation theory of $A_{n}^{d}$. This gives the other half of the picture from \cite{ot}, where it is shown how triangulations of even-dimensional cyclic polytopes arise in the representation theory of $A_{n}^{d}$.

We first show that triangulations of odd-dimensional cyclic polytopes correspond to equivalences classes of $d$-maximal green sequences, which we define as the analogues of maximal green sequences in higher Auslander--Reiten theory. In this section we are now using the cluster-tilting framework; in the tilting framework triangulations of odd-dimensional cyclic polytopes correspond to maximal sequences of left mutations of tilting modules.

If $\Lambda$ is a $d$-representation-finite $d$-hereditary algebra over a field $K$, then we define a \emph{$d$-maximal green sequence} for $\Lambda$ to be a sequence $(T_{1}, \dots, T_{r})$ of cluster-tilting objects in $\clus_{\Lambda}$ such that  $T_{1}=\Lambda$, $T_{r}=\Lambda[d]$, and, for $i \in [r-1]$, $T_{i+1}$ is a left mutation of $T_{i}$. Let $\mathcal{MG}_{d}(\Lambda)$ denote the set of $d$-maximal green sequences of $\Lambda$. Given a $d$-maximal green sequence $G$, we denote the set of isomorphism classes of indecomposable summands of cluster-tilting objects occurring in $G$ by $\Sigma(G)$. We define an equivalence relation on the set of $d$-maximal green sequences by $G \sim G'$ if and only if $\Sigma(G)=\Sigma(G')$. As before, we use $\widetilde{\mathcal{MG}}_{d}(\Lambda)$ to denote the set of equivalence classes of $\mathcal{MG}_{d}(\Lambda)$ under the relation $\sim$.

\begin{theorem}\label{thm:max_green}
Triangulations of the cyclic polytope $C(n+2d+1,2d+1)$ are in bijection with $\widetilde{\mathcal{MG}}_{d}(A_{n}^{d})$.
\end{theorem}
\begin{proof}[Sketch]
This follows from \cite[Theorem 1.1(ii)]{rambau}, which states that triangulations of $C(n + 2d + 1, 2d + 1)$ are in bijection with equivalence classes of maximal chains in $\mathcal{S}_{1}(n + 2d + 1, 2d)$ under the relation of differing by a permutation of bistellar flip operations. By the analogue of Theorem~\ref{thm:even:first} in the cluster-tilting framework, elements of $\mathcal{MG}(A_{n}^{d})$ correspond to maximal chains in $\mathcal{S}_{1}(n + 2d + 1,2d)$. Given two $d$-maximal green sequences $G, G'$ of $A_{n}^{d}$ corresponding to maximal chains $C, C'$ in $\mathcal{S}_{1}(n + 2d + 1, 2d)$, one can then show that $\Sigma(G) = \Sigma(G')$ if and only if $C$ and $C'$ give the same triangulation of $C(n + 2d + 1, 2d + 1)$.
\end{proof}

\begin{remark}
The number of $(2d + 1)$-simplices in a triangulation $\mathcal{T}$ of $C(n + 2d + 1, 2d + 1)$ is equal to the number of increasing bistellar flips in the corresponding chain in $\mathcal{S}_{1}(n + 2d + 1, 2d + 1)$, and hence equal to the length of any $d$-maximal green sequence in the corresponding equivalence class $[G]$. Unlike in even dimension, where the number of $2d$-simplices in a triangulation is fixed, in odd dimensions, the number of $(2d + 1)$-simplices in a triangulation of a given cyclic polytope is not fixed. This reflects the fact that $d$-maximal green sequences are not always of a fixed length for a given algebra.
\end{remark}

\subsection{First order}\label{odd-first}

We now explain our algebraic characterisation of the first higher Stasheff--Tamari order in terms of $d$-maximal green sequences. An \emph{oriented polygon} is a sub-poset of $\mathcal{S}_{1}(m,2d)$ formed of a union of a chain of covering relations of length $d + 2$ with a chain of covering relations of length $d + 1$, such that these chains intersect only at the top and bottom. For an illustration see Figure~\ref{fig:deformation}. Here the \emph{length} of a chain is the number of covering relations in it. We think of an oriented polygon as being oriented from the longer side to the shorter side. If two $d$-maximal green sequences $G,G'$ differ only in that $G$ contains the longer side of an oriented polygon and $G'$ contains the shorter side, then we say that $G'$ is \emph{an increasing elementary polygonal deformation} of $G$. Note that an increasing elementary polygonal deformation decreases the length of the sequence. The reader may like to think that an increasing elementary polygonal increases the \emph{speed} of the sequence.  Our terminology is based on \cite{hi-no-gap,g-mc}, although the notion of polygon we need differs from both of these.

\begin{theorem}\label{thm:odd:first}
Let $\mathcal{T}, \mathcal{T}' \in \mathcal{S}(n+2d+1,2d+1)$ correspond to equivalence classes of $d$-maximal green sequences $[G], [G'] \in \widetilde{\mathcal{MG}}(A_{n}^{d})$. Then $\mathcal{T}\lessdot_{1}\mathcal{T}'$ if and only if there are equivalence class representatives $\widehat{G} \in [G]$ and $\widehat{G}' \in [G']$ such that $\widehat{G}'$ is an increasing elementary polygonal deformation of $\widehat{G}$.
\end{theorem}
\begin{proof}[Sketch]
The main idea of the proof is as follows. Increasing bistellar flips of triangulations of $C(n + 2d + 1, 2d + 1)$ are induced by $(2d + 2)$-simplices. A $(2d + 2)$-simplex has $d + 2$ lower facets and $d + 1$ upper facets, each of which is a $(2d + 1)$-simplex. The $(2d + 1)$-simplices of a triangulation of $C(n + 2d + 1, 2d + 1)$ correspond to the increasing bistellar flips in the maximal chain in $\mathcal{S}_{1}(n + 2d + 1, 2d)$, which are the mutations in the $d$-maximal green sequence. Hence, what one needs to show for the forwards direction is that for every increasing bistellar flip of $\mathcal{T}$, one can find a corresponding maximal chain in $\mathcal{S}_{1}(n + 2d + 1, 2d)$ where all the lower facets of the $(2d + 2)$-simplex occur as consecutive mutations. For the backwards direction one needs to show that if one can replace a subchain of length $d + 2$ with a subchain of length $d + 1$ then these must correspond to the upper and lower facets of some $(2d + 2)$-simplex, which therefore gives an increasing bistellar flip.
\end{proof}

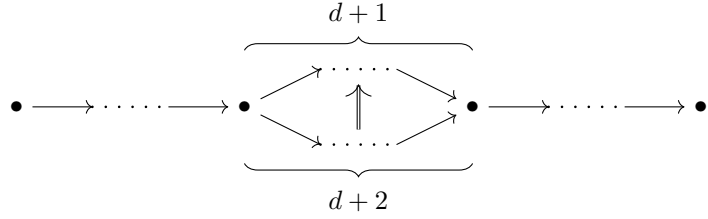
\begin{figure}
\caption{An increasing elementary polygonal deformation of $d$-maximal green sequences.}\label{fig:deformation}
\[\begin{tikzpicture}
\node (1) at (-1,0) {$\bullet$};
\coordinate (2f) at (0,0);
\coordinate (2s) at (1,0);
\node (3) at (2,0) {$\bullet$};
\coordinate (4u) at (3,0.5);
\coordinate (4d) at (3,-0.5);
\coordinate (5u) at (4,0.5);
\coordinate (5d) at (4,-0.5);
\node (6) at (5,0) {$\bullet$};
\coordinate (7f) at (6,0);
\coordinate (7s) at (7,0);
\node (8) at (8,0) {$\bullet$};


\draw[->] (1) -- (2f);
\draw[dot diameter=1pt, dot spacing=5pt, dots] (2f) -- (2s);
\draw[->] (2s) -- (3);
\draw[->] (3) -- (4u);
\draw[->] (3) -- (4d);
\draw[dot diameter=1pt, dot spacing=5pt, dots] (4u) -- (5u);
\draw[dot diameter=1pt, dot spacing=5pt, dots] (4d) -- (5d);
\draw[->] (5u) -- (6);
\draw[->] (5d) -- (6);
\draw[->] (6) -- (7f);
\draw[dot diameter=1pt, dot spacing=5pt, dots] (7f) -- (7s);
\draw[->] (7s) -- (8);


\draw[double, ->] (3.5, -0.3) -- (3.5, 0.3);


\draw[decorate,decoration={brace,amplitude=5pt}]
(2, 0.75) -- (5, 0.75);
\draw[decorate,decoration={brace,amplitude=5pt,mirror}]
(2, -0.75) -- (5, -0.75);


\node at (3.5, 1.25) {$d + 1$};
\node at (3.5, -1.25) {$d + 2$};
\end{tikzpicture}\]
\end{figure}

\begin{remark}
Basic cluster-tilting objects for $A_{n}^{d}$ are in bijection with triangulations of $C(n + 2d + 1, 2d)$ by \cite[Theorem 4.4]{ot}. By \cite[Theorem 3.4]{kv-poly}, the set of triangulations of $C(n + 2d + 1, 2d)$ forms an $n$-category, that is, a category enriched in $(n - 1)$-categories---where ordinary categories are 1-categories. Hence, one can impose the structure of an $n$-category on the set of (isomorphism classes of) basic cluster-tilting objects for $A_{n}^{d}$. Indeed, the irreducible 1-morphisms of this category are left mutations, and the irreducible 2-morphisms are the increasing elementary polygonal deformations from Theorem~\ref{thm:odd:first}.
\end{remark}

\subsection{Second order}\label{odd-second}

We now show how the second higher Stasheff--Tamari order may be interpreted in terms of equivalence classes of $d$-maximal green sequences.

\begin{theorem}\label{thm:odd:second}
Given two triangulations $\mathcal{T}, \mathcal{T}' \in \mathcal{S}(n+2d+1,2d+1)$ corresponding to equivalence classes of $d$-maximal green sequences $[G], [G'] \in \widetilde{\mathcal{MG}}(A_{n}^{d})$, then $\mathcal{T} \leqslant_{2} \mathcal{T}'$ if and only if $\Sigma(G) \supseteq  \Sigma(G')$.
\end{theorem}
\begin{proof}[Sketch]
This theorem requires another interpretation of the second higher Stasheff--Tamari order. We show that the second higher Stasheff--Tamari order is equivalent to reverse inclusion of supermersion sets of $\floor{\delta/2}$-simplices, where supermersion is defined dually to submersion. Note that this allows us to consider lower-dimensional simplices than submersion, which requires $\ceil{\delta/2}$-simplices. One can show that a $d$-supermersion set of a triangulation $\mathcal{T}$ of $C(n + 2d + 1, 2d + 1)$ is given by the $d$-simplices of the triangulation, since all $d$-simplices in $C(n + 2d + 1, 2d + 2)$ lie in the lower facets of the cyclic polytope. It suffices to consider the internal $d$-simplices of the triangulation, since $d$-simplices in the boundary facets belong to every triangulation. One can then show that internal $d$-simplices of $\mathcal{T}$ and $\mathcal{T}'$ correspond to the objects of $\Sigma(G)$ and $\Sigma(G')$ which are neither projectives nor shifted projectives. Since the projectives and shifted projectives are summands of every $d$-maximal green sequence, the supermersion set of $\mathcal{T}$ contains the supermersion set of $\mathcal{T}'$ if and only if $\Sigma(G) \supseteq \Sigma(G')$.
\end{proof}

\begin{remark}
The conjecture that the higher Stasheff--Tamari orders are equal in the odd-dimensional cases amounts to the existence of a series of increasing elementary polygonal deformations from $[G]$ to $[G']$ whenever we have $\Sigma(G) \supseteq \Sigma(G')$. This is similar to the ``no-gap'' conjecture of Br\"ustle, Dupont, and Perotin, which states that the set of lengths of maximal green sequences of an algebra has no gaps in it: if there is a series of increasing elementary polygonal deformations from $[G]$ to $[G']$, then there exists a $d$-maximal green sequence of every length between that of $G$ and that of $G'$. Cases of the no-gap conjecture were proven in \cite{g-mc,hi-no-gap} by showing how the set of maximal green sequences was connected by deformations across polygons. The difference here is that maximal green sequences connected by deformations across squares are always equivalent in our sense, and that we take account of the direction of deformations across other polygons.
\end{remark}

Since it is known for dimension 3 that the higher Stasheff--Tamari orders are equal and are lattices \cite[Theorem 4.9 and Theorem 4.10]{er}, we obtain the following corollary.

\begin{corollary}\label{cor-a-green-lat}
The two orders on $\widetilde{\mathcal{MG}}_{1}(A_{n})$ from Theorem~\ref{thm:odd:first} and Theorem~\ref{thm:odd:second} are equal and are lattices.
\end{corollary}

\begin{remark}
In independent work, \cite{gorsky_phd,gorsky_note,gorsky_forthcoming} Gorsky defines two orders on the set of equivalence classes of maximal green sequences of a Dynkin quiver, using a different framework, and proves that they are the same. For type $A$ quivers, these coincide with the two higher Stasheff--Tamari orders considered here in dimension 3.
\end{remark}

\begin{remark}\label{ex-counter-green-lat}
It is not in general true that the set of equivalence classes of maximal green sequences of a finite-dimensional algebra is a lattice under the orders we consider here. For example, the preprojective algebra of $A_{2}$ only has two maximal green sequences. These are not equivalent to each other, and nor are they related by either of the orders described above.

One might wonder whether the set of equivalence classes of maximal green sequences is a lattice for other hereditary algebras. However, computer calculations reveal that the set of equivalence classes of maximal green sequences of the path algebra of Dynkin type $D_{4}$ is not a lattice.
\end{remark}

\begin{example}
We illustrate Theorem~\ref{thm:max_green}, Theorem~\ref{thm:odd:first}, and Theorem~\ref{thm:odd:second} with the example $n = 2$, $d = 1$. The Auslander--Reiten quiver of the category $\clus_{A_{2}}$ is shown in Figure~\ref{fig:clus_ar_quiv}. There exist five cluster-tilting objects in $\clus_{A_{2}}$, which correspond to triangulations of $C(5, 2)$, as shown in Figure~\ref{fig:2d_triangs}. By Theorem~\ref{thm:max_green}, the two maximal green sequences formed from these cluster-tilting objects correspond to the two possible triangulations of $C(5, 3)$, as shown in Figure~\ref{fig:A2_mgs}. Note here how the mutations in each maximal green sequence correspond to the 3-simplices in the three-dimensional triangulation. Let the longer maximal green sequence here be $G$ and the shorter maximal green sequence be $G'$, with $\mathcal{T}$ and $\mathcal{T}'$ the corresponding triangulations. Then it can be seen from the figure that $G'$ is an increasing elementary polygonal deformation of $G$, which, by Theorem~\ref{thm:odd:first}, corresponds to the fact that $\mathcal{T}'$ is an increasing bistellar flip of $\mathcal{T}$. Moreover,
\begin{align*}
\Sigma(G) &= \left\{ \tcs{1},\, \tcs{2\\1},\, \tcs{2},\, \tcs{1}[1],\, \tcs{2\\1}[1] \right\},\\
\Sigma(G) &= \left\{ \tcs{1},\, \tcs{2\\1},\, \tcs{1}[1],\, \tcs{2\\1}[1] \right\},
\end{align*}
so that $\Sigma(G) \supseteq \Sigma(G')$. By Theorem~\ref{thm:odd:second}, we have that $\mathcal{T} \leqslant_{2} \mathcal{T}'$.

\begin{figure}
\caption{Maximal green sequences of $A_{2}$.}\label{fig:A2_mgs}
\vspace{1cm}
\[
\begin{tikzpicture}

\begin{scope}[shift={(0.2,0)},scale=0.9]



\begin{scope}[shift={(-3,0)}]


\coordinate(1) at (-2,2);
\coordinate(2) at (-1,0.5);
\coordinate(3) at (0,0);
\coordinate(4) at (1,0.5);
\coordinate(5) at (2,2);


\node(v1) at (1) {$\bullet$};
\node(v2) at (2) {$\bullet$};
\node(v3) at (3) {$\bullet$};
\node(v4) at (4) {$\bullet$};
\node(v5) at (5) {$\bullet$};


\node(mod2) at (0.5*-2 + 0.5*0, 0.5*2 + 0.5*0) {\color{red}$\tcs{1}$};
\node(mod12) at (0.5*-2 + 0.5*1, 0.5*2 + 0.5*0.5) {\color{red}$\tcs{2\\1}$};


\draw (1) -- (2) -- (3) -- (4) -- (5) -- (1);
\draw[blue] (1) -- (mod2) -- (3);
\draw[blue] (1) -- (mod12) -- (4);


\draw[ultra thick, OliveGreen, ->] (0,-0.5) -- (0,-1.5);

\end{scope}


\begin{scope}[shift={(-4,-1.5)}, scale=0.7]


\coordinate(1a) at (-2.75,2.25);
\coordinate(2a) at (-1.75,0.75);
\coordinate(3a) at (-0.75,0.25);
\coordinate(4a) at (0.25,0.75);


\draw (1a) --(4a);
\draw (1a) -- (2a);
\draw (2a) -- (3a);
\draw (3a) -- (4a);
\draw (2a) -- (4a);
\draw[dotted] (1a) -- (3a); 


\node at (1a) {$\bullet$};
\node at (2a) {$\bullet$};
\node at (3a) {$\bullet$};
\node at (4a) {$\bullet$};

\end{scope}


\begin{scope}[shift={(-3,-4)}]


\coordinate(1) at (-2,2);
\coordinate(2) at (-1,0.5);
\coordinate(3) at (0,0);
\coordinate(4) at (1,0.5);
\coordinate(5) at (2,2);


\node(v1) at (1) {$\bullet$};
\node(v2) at (2) {$\bullet$};
\node(v3) at (3) {$\bullet$};
\node(v4) at (4) {$\bullet$};
\node(v5) at (5) {$\bullet$};


\node(mod1) at (0.5*-1 + 0.5*1, 0.5*0.5 + 0.5*0.5) {\color{red}$\tcs{2}$};
\node(mod12) at (0.5*-2 + 0.5*1, 0.5*2 + 0.5*0.5) {\color{red}$\tcs{2\\1}$};


\draw (1) -- (2) -- (3) -- (4) -- (5) -- (1);
\draw[blue] (2) -- (mod1) -- (4);
\draw[blue] (1) -- (mod12) -- (4);


\draw[ultra thick, OliveGreen, ->] (0,-0.5) -- (0,-1.5);

\end{scope}


\begin{scope}[shift={(-5,-6)}, scale=0.7]


\coordinate(1b) at (-1.75,2.5);
\coordinate(2b) at (-0.75,1);
\coordinate(4b) at (1.15,1);
\coordinate(5b) at (2.25,2.5);


\draw[fill=gray!30, draw=none] (1b) -- (2b) -- (5b) -- (1b);


\draw (1b) -- (5b);
\draw (1b) -- (2b);
\draw (2b) -- (4b);
\draw (4b) -- (5b);
\draw (2b) -- (5b);
\draw[dotted] (1b) -- (4b);


\node at (1b) {$\bullet$};
\node at (2b) {$\bullet$};
\node at (4b) {$\bullet$};
\node at (5b) {$\bullet$};

\end{scope}


\begin{scope}[shift={(-3,-8)}]


\coordinate(1) at (-2,2);
\coordinate(2) at (-1,0.5);
\coordinate(3) at (0,0);
\coordinate(4) at (1,0.5);
\coordinate(5) at (2,2);


\node(v1) at (1) {$\bullet$};
\node(v2) at (2) {$\bullet$};
\node(v3) at (3) {$\bullet$};
\node(v4) at (4) {$\bullet$};
\node(v5) at (5) {$\bullet$};


\node(mod1) at (0.5*-1 + 0.5*1, 0.5*0.5 + 0.5*0.5) {\color{red}$\tcs{2}$};
\node(mod21) at (0.5*-1 + 0.5*2, 0.5*0.5 + 0.5*2) {\color{red}$\tcs{1}[1]$};


\draw (1) -- (2) -- (3) -- (4) -- (5) -- (1);
\draw[blue] (2) -- (mod1) -- (4);
\draw[blue] (2) -- (mod21) -- (5);


\draw[ultra thick, OliveGreen, ->] (0,-0.5) -- (0,-1.5);

\end{scope}


\begin{scope}[shift={(-6,-9.5)}, scale=0.7]


\coordinate(2c) at (-0.25,0.25);
\coordinate(3c) at (0.75,-0.25);
\coordinate(4c) at (1.75,0.25);
\coordinate(5c) at (2.75,1.75);


\draw[fill=gray!30, draw=none] (2c) -- (5c) -- (3c) -- (2c);
\draw[fill=gray!30, draw=none] (3c) -- (4c) -- (5c) -- (3c);


\draw (2c) -- (5c);
\draw (2c) -- (3c);
\draw (3c) -- (4c);
\draw (4c) -- (5c);
\draw (3c) -- (5c);
\draw[dotted] (2c) -- (4c);


\node at (2c) {$\bullet$};
\node at (3c) {$\bullet$};
\node at (4c) {$\bullet$};
\node at (5c) {$\bullet$};

\end{scope}


\begin{scope}[shift={(-3,-12)}]


\coordinate(1) at (-2,2);
\coordinate(2) at (-1,0.5);
\coordinate(3) at (0,0);
\coordinate(4) at (1,0.5);
\coordinate(5) at (2,2);


\node(v1) at (1) {$\bullet$};
\node(v2) at (2) {$\bullet$};
\node(v3) at (3) {$\bullet$};
\node(v4) at (4) {$\bullet$};
\node(v5) at (5) {$\bullet$};


\node(mod21) at (0.7*-1 + 0.3*2, 0.7*0.5 + 0.3*2) {\color{red}$\tcs{1}[1]$};
\node(mod121) at (0.4*2 + 0.6*0, 0.4*2 + 0.6*0) {\color{red}$\tcs{2\\1}[1]$};


\draw (1) -- (2) -- (3) -- (4) -- (5) -- (1);
\draw[blue] (2) -- (mod21) -- (5);
\draw[blue] (3) -- (mod121) -- (5);

\end{scope}


\begin{scope}[shift={(-3,-15)}]


\coordinate(1a) at (-2.77,2.25);
\coordinate(2a) at (-1.77,0.75);
\coordinate(3a) at (-0.77,0.25);
\coordinate(4a) at (0.23,0.75);

\coordinate(1b) at (-1.75,2.5);
\coordinate(2b) at (-0.75,1);
\coordinate(4b) at (1.15,1);
\coordinate(5b) at (2.25,2.5);

\coordinate(2c) at (-0.25,0.25);
\coordinate(3c) at (0.75,-0.25);
\coordinate(4c) at (1.75,0.25);
\coordinate(5c) at (2.75,1.75);


\draw[fill=gray!30, draw=none] (1b) -- (2b) -- (5b) -- (1b);
\draw[fill=gray!30, draw=none] (2c) -- (5c) -- (3c) -- (2c);
\draw[fill=gray!30, draw=none] (3c) -- (4c) -- (5c) -- (3c);



\path[name path = line 1] (1a) -- (4a);


\path[name path = line 2] (1b) -- (2b);
\path[name path = line 3] (2b) -- (4b);


\path [name intersections={of = line 1 and line 2}];
\coordinate (a)  at (intersection-1);

\path [name intersections={of = line 1 and line 3}];
\coordinate (b)  at (intersection-1);


\draw (1a) -- (a);
\draw[dotted] (a) -- (b);
\draw (b) -- (4a);


\draw (1a) -- (2a);
\draw (2a) -- (3a);
\draw (3a) -- (4a);
\draw (2a) -- (4a);
\draw[dotted] (1a) -- (3a); 


\draw (1b) -- (5b);
\draw (1b) -- (2b);
\draw (2b) -- (4b);
\draw (4b) -- (5b);
\draw (2b) -- (5b);
\draw[dotted] (1b) -- (4b);


\draw (2c) -- (5c);
\draw (2c) -- (3c);
\draw (3c) -- (4c);
\draw (4c) -- (5c);
\draw (3c) -- (5c);
\draw[dotted] (2c) -- (4c);


\node at (1a) {$\bullet$};
\node at (2a) {$\bullet$};
\node at (3a) {$\bullet$};
\node at (4a) {$\bullet$};

\node at (1b) {$\bullet$};
\node at (2b) {$\bullet$};
\node at (4b) {$\bullet$};
\node at (5b) {$\bullet$};

\node at (2c) {$\bullet$};
\node at (3c) {$\bullet$};
\node at (4c) {$\bullet$};
\node at (5c) {$\bullet$};

\end{scope}

\end{scope}


\begin{scope}[shift={(-0.2,0)},scale=0.9]



\begin{scope}[shift={(3,0)}]


\coordinate(1) at (-2,2);
\coordinate(2) at (-1,0.5);
\coordinate(3) at (0,0);
\coordinate(4) at (1,0.5);
\coordinate(5) at (2,2);


\node(v1) at (1) {$\bullet$};
\node(v2) at (2) {$\bullet$};
\node(v3) at (3) {$\bullet$};
\node(v4) at (4) {$\bullet$};
\node(v5) at (5) {$\bullet$};


\node(mod2) at (0.5*-2 + 0.5*0, 0.5*2 + 0.5*0) {\color{red}$\tcs{1}$};
\node(mod12) at (0.5*-2 + 0.5*1, 0.5*2 + 0.5*0.5) {\color{red}$\tcs{2\\1}$};


\draw (1) -- (2) -- (3) -- (4) -- (5) -- (1);
\draw[blue] (1) -- (mod2) -- (3);
\draw[blue] (1) -- (mod12) -- (4);


\draw[ultra thick, OliveGreen, ->] (0,-0.5) -- (0,-3.5);

\end{scope}


\begin{scope}[shift={(4.5,-3)}, scale=0.7]


\coordinate(31b) at (-1.5,2.25);
\coordinate(33b) at (0.5,0.25);
\coordinate(34b) at (1.5,0.75);
\coordinate(35b) at (2.5,2.25);


\draw[fill=gray!30, draw=none] (33b) -- (34b) -- (35b) -- (33b);


\draw (31b) --(33b);
\draw (33b) -- (34b);
\draw (34b) -- (35b);
\draw (35b) -- (31b);
\draw (33b) -- (35b);
\draw[dotted] (31b) -- (34b);


\node at (31b) {$\bullet$};
\node at (33b) {$\bullet$};
\node at (34b) {$\bullet$};
\node at (35b) {$\bullet$};

\end{scope}


\begin{scope}[shift={(3,-6)}]


\coordinate(1) at (-2,2);
\coordinate(2) at (-1,0.5);
\coordinate(3) at (0,0);
\coordinate(4) at (1,0.5);
\coordinate(5) at (2,2);


\node(v1) at (1) {$\bullet$};
\node(v2) at (2) {$\bullet$};
\node(v3) at (3) {$\bullet$};
\node(v4) at (4) {$\bullet$};
\node(v5) at (5) {$\bullet$};


\node(mod2) at (0.5*-2 + 0.5*0, 0.5*2 + 0.5*0) {\color{red}$\tcs{1}$};
\node(mod121) at (0.4*2 + 0.6*0, 0.4*2 + 0.6*0) {\color{red}$\tcs{2\\1}[1]$};


\draw (1) -- (2) -- (3) -- (4) -- (5) -- (1);
\draw[blue] (1) -- (mod2) -- (3);
\draw[blue] (3) -- (mod121) -- (5);


\draw[ultra thick, OliveGreen, ->] (0,-0.5) -- (0,-3.5);

\end{scope}


\begin{scope}[shift={(5.5,-9)}, scale=0.7]


\coordinate(31f) at (-2.5,1.75);
\coordinate(32f) at (-1.5,0.2);
\coordinate(33f) at (-0.5,-0.25);
\coordinate(35f) at (1.5,1.75);


\draw[fill=gray!30, draw=none] (31f) -- (32f) -- (35f) -- (31f);
\draw[fill=gray!30, draw=none] (32f) -- (33f) -- (35f) -- (32f);


\draw (31f) -- (32f);
\draw (32f) -- (33f);
\draw (33f) -- (35f);
\draw (35f) -- (31f);
\draw (32f) -- (35f);
\draw[dotted] (31f) -- (33f); 


\node at (31f) {$\bullet$};
\node at (32f) {$\bullet$};
\node at (33f) {$\bullet$};
\node at (35f) {$\bullet$};

\end{scope}


\begin{scope}[shift={(3,-12)}]


\coordinate(1) at (-2,2);
\coordinate(2) at (-1,0.5);
\coordinate(3) at (0,0);
\coordinate(4) at (1,0.5);
\coordinate(5) at (2,2);


\node(v1) at (1) {$\bullet$};
\node(v2) at (2) {$\bullet$};
\node(v3) at (3) {$\bullet$};
\node(v4) at (4) {$\bullet$};
\node(v5) at (5) {$\bullet$};


\node(mod21) at (0.7*-1 + 0.3*2, 0.7*0.5 + 0.3*2) {\color{red}$\tcs{1}[1]$};
\node(mod121) at (0.4*2 + 0.6*0, 0.4*2 + 0.6*0) {\color{red}$\tcs{2\\1}[1]$};


\draw (1) -- (2) -- (3) -- (4) -- (5) -- (1);
\draw[blue] (2) -- (mod21) -- (5);
\draw[blue] (3) -- (mod121) -- (5);

\end{scope}


\begin{scope}[shift={(3,-15)}]


\coordinate(31f) at (-2.5,1.75);
\coordinate(32f) at (-1.5,0.2);
\coordinate(33f) at (-0.5,-0.25);
\coordinate(35f) at (1.5,1.75);

\coordinate(31b) at (-1.5,2.25);
\coordinate(33b) at (0.5,0.25);
\coordinate(34b) at (1.5,0.75);
\coordinate(35b) at (2.5,2.25);



\path[name path = line 2] (33f) -- (35f);
\path[name path = line 1] (35f) -- (31f);


\path[name path = line 3] (31b) -- (33b);


\path [name intersections={of = line 1 and line 3}];
\coordinate (a)  at (intersection-1);

\path [name intersections={of = line 2 and line 3}];
\coordinate (b)  at (intersection-1);


\draw[fill=gray!30, draw=none] (31f) -- (32f) -- (35f) -- (31f);
\draw[fill=gray!30, draw=none] (32f) -- (33f) -- (35f) -- (32f);
\draw[fill=gray!30, draw=none] (33b) -- (34b) -- (35b) -- (33b);



\draw (31b) -- (a);
\draw[dotted] (a) -- (b);
\draw (b) -- (33b);


\draw (31f) -- (32f);
\draw (32f) -- (33f);
\draw (33f) -- (35f);
\draw (35f) -- (31f);
\draw (32f) -- (35f);
\draw[dotted] (31f) -- (33f); 


\draw (33b) -- (34b);
\draw (34b) -- (35b);
\draw (35b) -- (31b);
\draw (33b) -- (35b);
\draw[dotted] (31b) -- (34b);


\node at (31f) {$\bullet$};
\node at (32f) {$\bullet$};
\node at (33f) {$\bullet$};
\node at (35f) {$\bullet$};
\node at (31b) {$\bullet$};
\node at (33b) {$\bullet$};
\node at (34b) {$\bullet$};
\node at (35b) {$\bullet$};

\end{scope}

\end{scope}

\end{tikzpicture}
\]
\end{figure}
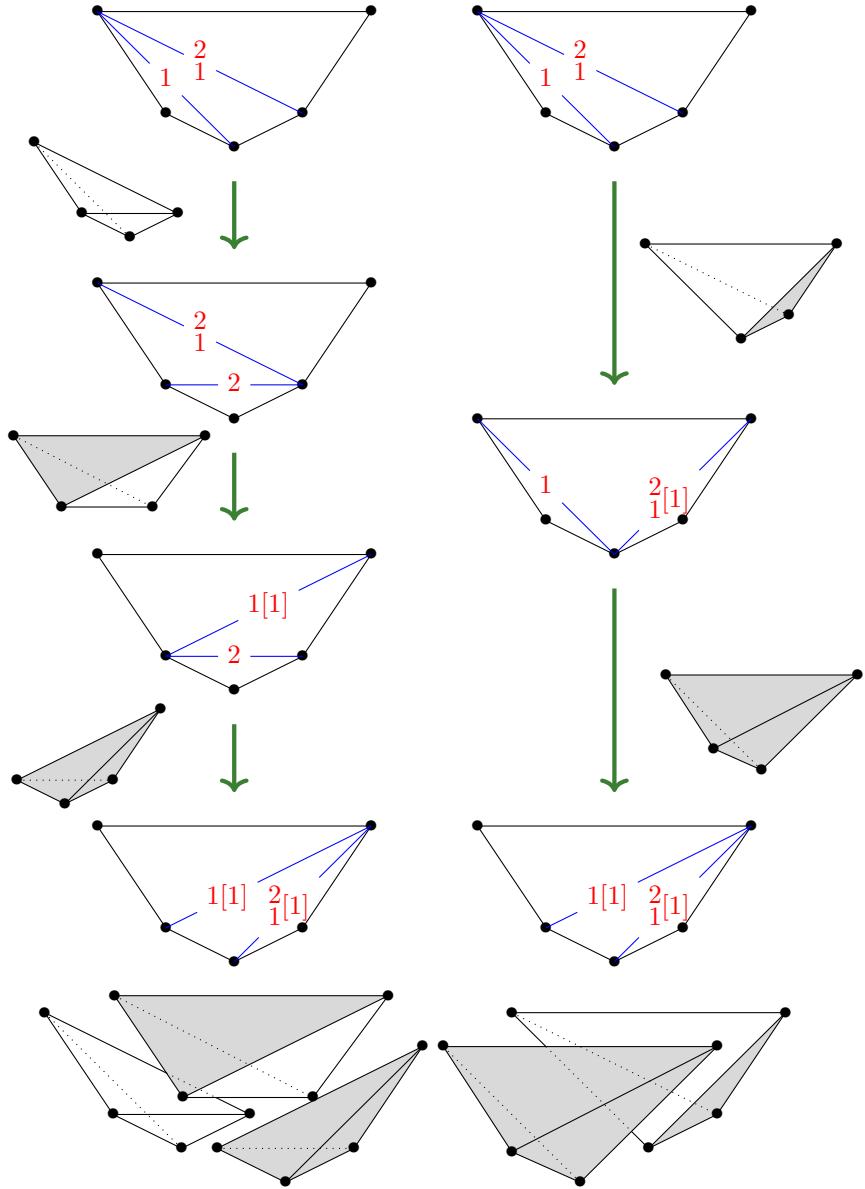
\end{example}

\printbibliography

\end{document}